\documentclass[12pt]{amsbook}
\usepackage{amsmath, enumerate, amssymb, latexsym,mathrsfs,amssymb,graphicx}

\usepackage{epstopdf}

\usepackage{prettyref}

\usepackage[
colorlinks=true,
linkcolor=blue,
 citecolor=blue,
  urlcolor=blue,
    pagebackref,
]{hyperref}

\renewcommand{\le}{\leqslant}
\renewcommand{\ge}{\geqslant}

\let \la=\lambda
\let \La=\Lambda
\let \e=\varepsilon
\let \d=\delta
\let \o=\omega
\let \a=\alpha

\let \b=\beta

\let \O=\Omega

\let \G=\Gamma

\newcommand{\s}{\mathcal S}
\newcommand{\md}{\mathscr D}

\newcommand{\wt}{\widetilde}

\newtheorem{theorem}{Theorem}[section]

\newtheorem{lemma}[theorem]{Lemma}

\theoremstyle{definition}
\newtheorem{definition}[theorem]{Definition}

\theoremstyle{remark}
\newtheorem{remark}[theorem]{Remark}

\numberwithin{equation}{section}

\begin{document}


\begin{center}
{\bf\Huge
Intuitive dyadic calculus: the basics}
\end{center}
\vskip 1cm
\begin{center}
{\Large Andrei K. Lerner and Fedor Nazarov}
\end{center}
\vskip 2cm







\begin{figure}[ht]
\begin{center}
\includegraphics[width=12cm,height=11cm]{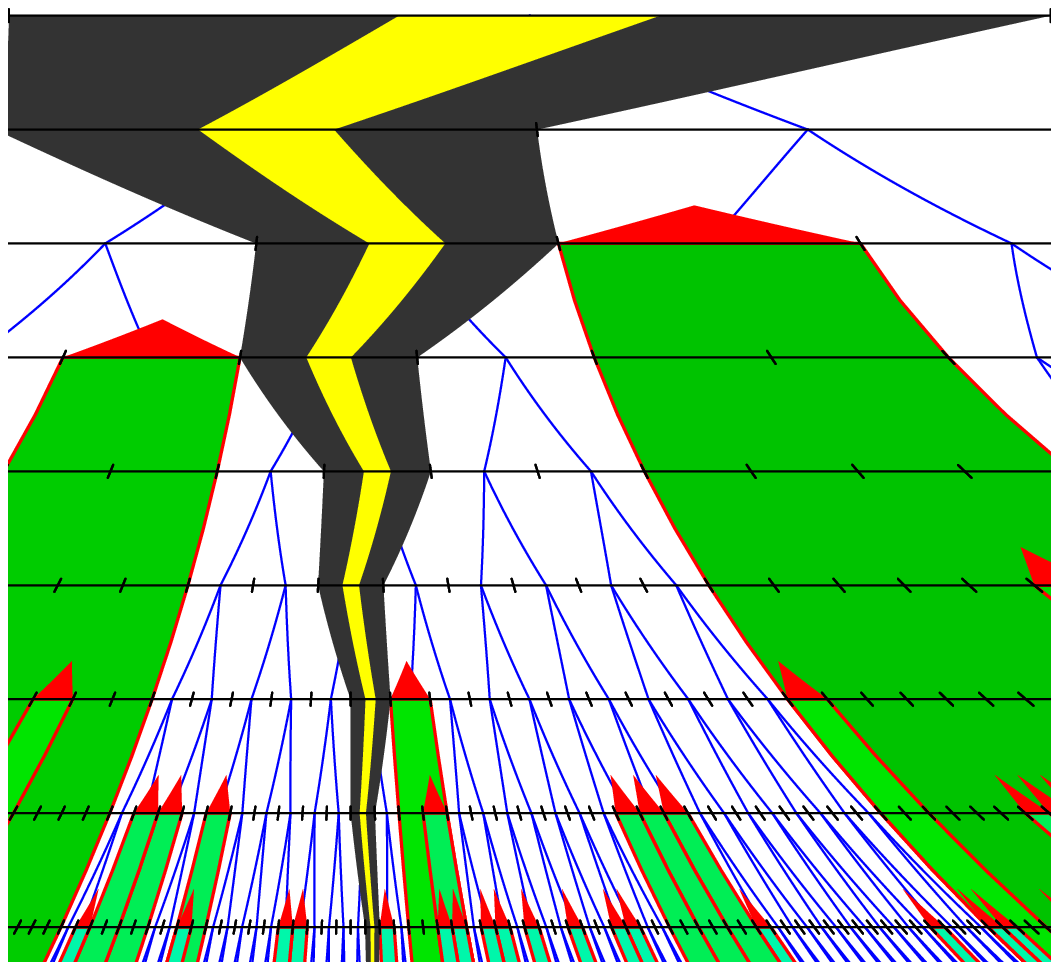}
\label{lightningbolt}
\end{center}
\end{figure}

\thispagestyle{empty}

\tableofcontents

\begin{center}
{\bf Foreword}
\end{center}
\rightline{\it What is the use of a book without pictures or conversations?}
\medskip

The current language of mathematics is primarily object-oriented. A classical textbook on a given subject
takes some notion, such as that of a singular integral operator or an abelian group and tries to say as much as possible about it
using whatever tools are needed for each particular lemma or proposition. The hope seems to be that the reader
will be fascinated enough with the wide landscape that slowly unfolds before his eyes to learn to jump over
trenches, climb trees and slopes, and wade through thickets and streams essentially on his own after watching his guide doing it effortlessly and with elegance. This certainly does make sense and most of us have learned mathematics exactly in this way.

However, there are at least a few of us, who look at the marble palaces of mathematics not from the viewpoint
of an artist or a town architect, but from that of a mason, a carpenter, or an engineer. When looking at an equine
statue, such a person marvels not so much at the flaring nostrils of the steed or the majestic posture of the rider, but rather at the amazing balance that allows the whole heavy composition to rest on its foundation touching it only by the two rear hoofs of the horse.

If you look at mathematics, or anything else, from this point of view, another landscape will open to your eyes:
not that of towers rising to the sky, wide streets and blossoming gardens, but that of the intricate machinery that makes
the whole city tick in unison: water pumps, electric lines, networks of tunnels and bridges, and so on up to
the concave shape of the pavements and the sewer ditches around them.

The description of this second landscape in an object-oriented language is pretty dull. Indeed, what can be said about a sewer ditch as an isolated object that is interesting enough to merit a discussion? Its beauty reveals itself only in the way it works in a heavy rain keeping the glossy boulevards and the shiny buildings around it safe from floods that would otherwise wash away both the speedy chariots and the noble inhabitants. In other words, its whole tale is in the way it operates, so a function-oriented language would be the best one to employ in this case. The textbooks written in this language are few and we decided that it might make sense to try our own hand at writing one.

Another common feature of many modern textbooks is their encyclopedic character. The authors probably feel like they are obliged to squeeze the whole body of human knowledge into a single volume. The resulting 400+ page leviathans are perfect reference books when one needs to check if some particular statement is true or known, but if one wants to taste the flavor of the subject for the first time, the sheer magnitude of the proposed endeavor of reading a few hundred pages of (usually quite laconic) text can easily scare him away.

Hence, instead of trying to present everything we knew about the dyadic techniques, we tried to hang a thin
thread from the very basics to one of the most recent results about the weighted norm inequalities for singular integral operators. If these words sound a bit intimidating, there is no need to worry. If the reader assumes that all functions and kernels are actually continuous in the entire space and then uses the usual Riemann integration theory instead of the Lebesgue one, he will lose next to nothing.

The resulting text of approximately 50 pages is as self-contained as it can be and we hope that the reader may be tempted to take
a look at something this short. It is also worth mentioning that quite a few of those 50 pages are devoted to explanations of why this or that
should be done in some or other particular way. Without these explanations and the illustrations this book would shrink even further.

We are aware of the criticism that each person visualizes any particular argument in his own way and that the commonly
used language that carefully avoids any informal passages is designed this way precisely to eliminate the superposition of any
subjective views or associations onto the pristinely objective chains of syllogisms. However, we have never seen those pristine chains anywhere.
The way every single definition or statement is composed already reveals if not the personality
of the author, then at least the school to which he belongs. Most important, the author's personality is inseparable
from the way he puts the whole argument together, so the choice is really between presenting a highly personal and subjective
structure of the proof without any labels or with labels explaining what,  in the opinion of the author, each joint or block is there for.

Of course, when looking at works of a highly skilled craftsman with ultimately weird internal logic, the first option may be preferable, but in general to strip the labels away is easier than to restore them. We thus felt free to supplement the formal
arguments with informal comments any time we wanted. We also included pictures and did not hesitate to change the existing terminology slightly or to introduce a new one when we thought that it might help to convey some particular point more efficiently.

The title comes partly as a challenge to some modern calculus textbook writers whose 3-pound bricks, all coming out of the same kiln at the speed of a Ford car plant conveyor belt, would constitute a decent weapon in a fray\footnote{if not for the glossy covers that slip out of one's fingers too easily, though
cutting a hand-size hole near one of the edges can partially remedy the issue and make the book easier to carry around in
times of peace} but are of little assistance in a classroom  (our hats are off for
the likes of M. Spivak, though).
The main problem is that those textbooks teach just to follow standard algorithms, thus viewing a student as a (rather badly wired and slow)
computer to program, never appealing to such natural human feelings as curiosity and admiration of the unfamiliar.
While in many these feelings are long dead and difficult to resurrect, the flickering light of the magic lamp of human wit and intelligence is not yet totally obscured by the infernal fires of regional wars or extinguished by the brainwashing waterfalls of mass media. We wanted to add some oil to that lamp. Whether and to what extent we have succeeded or failed is left to the judgement of the reader.

\medskip
\rightline{\textit{Summer 2014, Kent--Ramat Gan}}

\newpage
\section{Introduction}
The dyadic technique is a game of cubes, and this is the way we try to present it. We start the general theory with the basic
notion of a dyadic lattice, proceed with the multiresolution, the Three Lattice Theorem (probably due to T. Hyt\"onen),
augmentation and stopping times, and finish the  exposition with Carleson families.

The point of view we are trying to promote differs from the standard one in only one respect:
we try to exhibit various standard families of cubes as whole game sets, which can be reshuffled
and complemented freely in the related constructions as the need arises, rather than defining them recursively and going from one generation to another,
checking the properties by (backward) induction. This point of view led us to the introduction of an explicit graph structure
on a dyadic lattice and a couple of related notions, which are almost always mentioned in informal conversations
between specialists under various names, but had hardly ever been formalized in writing.

The particular application we chose to demonstrate how the dyadic cubes can be played with in analysis is weighted norm
inequalities for singular integral operators. The high points of this second half of the book are the estimate for an arbitrary
measurable function in terms of its $\la$-oscillations on a Carleson family of cubes,
a pointwise bound that allows one to easily reduce the weighted inequalities
for Calder\'on-Zygmund operators to those for sparse operators,
and a unified approach to the linear and multilinear theory.

The (small) cost to pay
is that the formalization of the multiweight problem most suitable for our exposition is somewhat different from the
commonly used one, though the translation from our language to the standard one is immediate, and we state the main result in this
book (the ``$A_p$-conjecture" about the sharp dependence of the operator norms in weighted $L^p$ spaces on the
joint Muckenhoupt norm of weights) in both forms in the final sections.

The history of dyadic techniques and weighted norm inequalities merits a separate volume. If such a monograph were ever
written  it would probably be longer than our entire exposition even if restricted to stating merely who did what
and when, omitting all subtle interplays of events and ideas.
We highlight a few relevant works most directly connected with our presentation
 in the short Historical Notes section near the end of this book,
but we are very far from claiming that our selection is complete or even representative.

\medskip
{\bf Acknowledgements}.
We are thankful to numerous people who read the preliminary version of this manuscript and shared their
remarks and suggestions with us. This project would also be difficult to carry out without the genereous support of the Israel Science Foundation and the National Science Foundation 
\footnote{ISF grant 953/13 (A.L.) and NSF grant DMS 080243 (F.N.)}.

\section{Dyadic cubes and lattices}
By a cube in ${\mathbb R}^n$ of sidelength $l$ we always mean a half-open cube
$$Q=[x_1,x_1+l)\times\dots\times[x_n,x_n+l)\quad(x=(x_1,\dots,x_n)\in {\mathbb R}^n,l>0)$$
with sides parallel to the coordinate axes. We will call the point $x$ the ``corner"
of the cube $Q$ and denote it by $x(Q)$. It will be also convenient to introduce the notation $c(Q)$
for the ``center" $x+\frac{1}{2}(l,\dots,l)$ of $Q$ and $\ell_Q$ for the sidelength of $Q$.

Let $Q$ be any cube in ${\mathbb R}^n$. A (dyadic) child of $Q$ is any of the $2^n$ cubes obtained by partitioning $Q$ by $n$
``median hyperplanes" (i.e, the hyperplanes parallel to the faces of $Q$ and dividing each edge into 2 equal parts).

\begin{figure}[ht]
\begin{center}
\includegraphics[width=5cm,height=5cm]{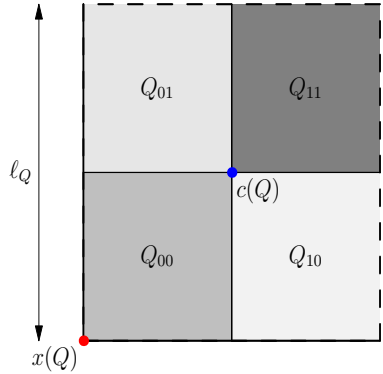}
\caption[]{A square $Q$ and its 4 dyadic children $Q_{ij}$.}
\label{dyadiccube}
\end{center}
\end{figure}

Passing from $Q$ to its children, then to the children of the children, etc., we obtain a standard dyadic lattice ${\mathcal D}(Q)$
of subcubes of $Q$ (see Figure \ref{DofQ}).

\begin{figure}[ht]
\begin{center}
\includegraphics[width=6.5cm,height=7cm]{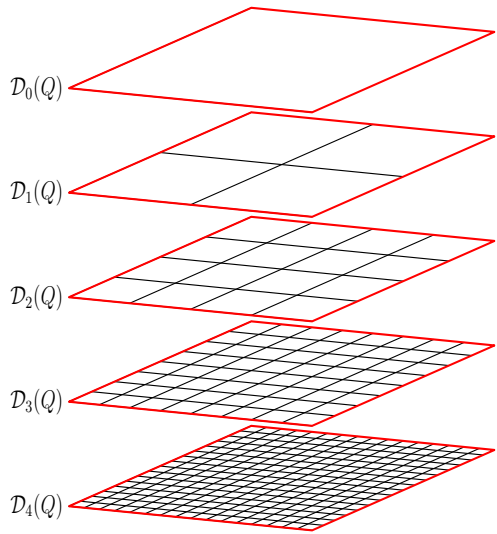}
\caption[]{The top 5 generations in the dyadic lattice $\mathcal D(Q)$.}
\label{DofQ}
\end{center}
\end{figure}

The cubes in this lattice enjoy several nice properties of which the most important ones seem to be

\medskip

\begin{enumerate}
\renewcommand{\labelenumi}{(\roman{enumi})}
\item
for each $k=0,1,2,\dots$, the cubes in the $k$-th generation ${\mathcal D}_k(Q)$ have the same sidelength $2^{-k}\ell_Q$
and tile $Q$ in a regular way, i.e., in the same way as the integer shifts of $[0,1)^n$ tile $\mathbb R^n$;
\item
each cube $Q'\in {\mathcal D}_k(Q)$ has $2^n$ children in ${\mathcal D}_{k+1}(Q)$ contained in it and (unless it is $Q$ itself)
one parent in ${\mathcal D}_{k-1}(Q)$ containing it;
\item
for every two cubes $Q',Q''\in {\mathcal D}(Q)$, either $Q'\cap Q''=\varnothing$, or $Q'\subset Q''$, or $Q''\subset Q'$;
\item
if $Q'\in {\mathcal D}(Q)$, then ${\mathcal D}(Q')\subset {\mathcal D}(Q)$.
\end{enumerate}

\medskip

The dyadic lattice ${\mathcal D}(Q)$ has many other advantages but two essential drawbacks: it is completely rigid and covers only
a  part of the entire space. All that can be done to compensate for these drawbacks is to move
it around and scale, which is enough for most purposes but still forces one to use awkward phrases
like ``considering only functions with
compact support and choosing $Q$ so that it covers the support" or ``assuming that $\int f=0$ to avoid
writing the top ``mean" term
in the Haar decomposition of $f$ separately".

Our first goal will be to introduce a notion of a dyadic lattice ${\mathscr D}$ that takes care of all ${\mathbb R}^n$ at once. Our particular choice of axioms
was dictated by the necessity to ensure that certain constructions be possible and the desire to make the property
of being a lattice reasonably easy to use or verify.
The cost is that the ``classical'' dyadic lattice
$${\mathcal D}=\{2^{-k}([m_1,m_1+1)\times\dots\times[m_n,m_n+1)): k\in {\mathbb Z}, m_1,\dots,m_n\in {\mathbb Z}\}$$
is not a dyadic lattice in our sense, though any finite set of cubes in ${\mathcal D}$ is a subset of some (actually, infinitely many) dyadic lattices
${\mathscr D}$.

\begin{definition}\label{dyadlat}
A dyadic lattice ${\mathscr D}$ in ${\mathbb R}^n$ is any collection of cubes such that
\begin{enumerate}[{(DL-1)}]
\item
if $Q\in\md$, then each child of $Q$ is in $\md$ as well (this, of course, implies immediately that ${\mathcal D}(Q)\subset\md$);
\item
every 2 cubes $Q',Q''\in \md$ have a common ancestor, i.e., there exists $Q\in\md$ such that $Q',Q''\in {\mathcal D}(Q)$;
\item
for every compact set $K\subset {\mathbb R}^n$, there exists a cube $Q\in \md$ containing $K$.
\end{enumerate}
\end{definition}

Property (DL-2) can be used several times in a row to find a common ancestor of any finite family of cubes by first finding an ancestor
$\widetilde Q_2$ of $Q_1$ and $Q_2$, then an ancestor $\widetilde Q_3$ of $\widetilde Q_2$ and $Q_3$, etc.
It also ensures that all cubes in $\md$ lie in a usual neat way with respect to each other, so all useful properties of cubes in ${\mathcal D}(Q)$
hold in $\md$ as well except the existence of the ``very top cube", which is usually a nuisance rather than an asset.

We can also split the cubes in $\md$ into generations $\md_k$ ($k\in {\mathbb Z}$) by choosing an arbitrary cube $Q_0\in \md$
and declaring it of generation $0$. The generation of any other cube $Q$ can be then determined by finding some common ancestor $P$ of
$Q$ and $Q_0$ and taking the difference of the generations of $Q$ and $Q_0$ in ${\mathcal D}(P)$.

It follows immediately from the way the generations are defined that all cubes in $\md_k$ have the same sidelength
$2^{-k}\ell_{Q_0}$. Property (DL-3) and the common ancestor trick also imply that the cubes in $\md_k$
tile ${\mathbb R}^n$.

We always think that the generation number increases as the cubes shrink and decreases as they expand, so the parent of $Q\in\md_k$ is in $\md_{k-1}$
and the children of $Q$ are in $\md_{k+1}$.

Once the definition is given, our first task is to demonstrate that dyadic lattices do exist. It is not hard and the idea is the same as
in the construction of the classical dyadic lattice: take any cube $Q$, construct the dyadic lattice ${\mathcal D}(Q)$, and then expand it
inductively up and sideways by choosing one of $2^n$ possible parents for the top cube and including it into $\md$ together with all its dyadic
subcubes during each step. The only catch is that the choice should be done either carefully enough, or recklessly enough (both approaches work)
to ensure that we do not just move in one direction all the time, like in the classical lattice where we always expand the top cube from its corner.
The particular way of expansion we will use is alternating the expansions from the corner with those from the vertex opposite to the corner.

\begin{figure}[ht]
\begin{center}
\includegraphics[width=11cm,height=5cm]{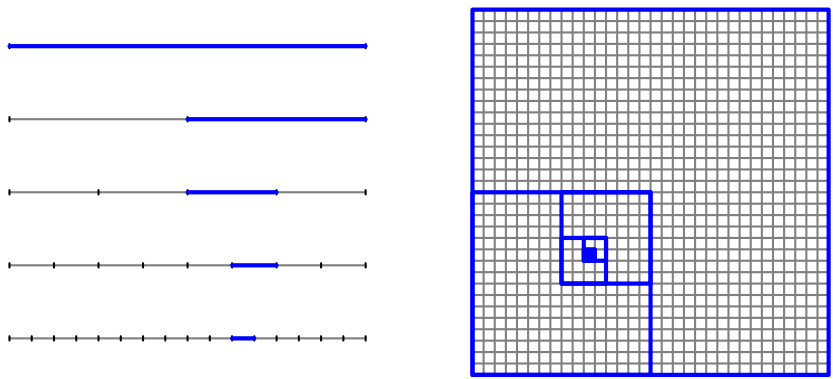}
\caption[]{The expanding families of intervals and squares.}
\label{wayup}
\end{center}
\end{figure}

This is by no means the only way to go. Starting with any cube $Q_0$ and the corresponding dyadic lattice ${\mathcal D}(Q_0)$, construct
any ascending sequence of cubes $Q_0\subset Q_1\subset Q_2\subset\dots$ so that $Q_j$ is a dyadic child of $Q_{j+1}$ and $\cup_j\text{int}Q_j={\mathbb R}^n$. Note that then $Q_j\in {\mathcal D}(Q_{j+1})$ and, thereby ${\mathcal D}(Q_{j})\subset {\mathcal D}(Q_{j+1})$. Put
${\mathscr D}=\cup_{j=1}^{\infty}{\mathcal D}(Q_{j})$. Then $\md$ satisfies all conditions of Definition \ref{dyadlat}.
\medskip
\begin{enumerate}[{(DL-1)}]
\item
If $Q\in\md$, then $Q\in {\mathcal D}(Q_j)$ for some $j$, so all children of $Q$ are in the same ${\mathcal D}(Q_j)\subset\md$.
\item
Let $Q',Q''\in \md$. Then $Q'\in{\mathcal D}(Q_{j'})$ and $Q''\in{\mathcal D}(Q_{j''})$ for some $j',j''$. The cube $Q_{\max(j',j'')}\in\md$ is a common ancestor of
$Q',Q''$.
\item
$\text{int}Q_j$ is an expanding chain of open sets covering the entire space, so any compact will get absorbed sooner or later.
\end{enumerate}

\medskip

The chain depicted on Figure \ref{wayup} is
$$
Q_j = \begin{cases} [-2^j,2^{j+1})^n, & j\,\text{is odd}; \\
[-2^{j+1},2^j)^n,& j\,\text{is even}. \end{cases}
$$
It clearly satisfies all requirements, but most other sequences will do just as well.

\subsection{Dyadic lattice as a multiresolution}
Any dyadic lattice $\md$ consists of a sequence of generations $\md_k$. The cubes in $\md_k$ form
a regular tiling $T_k$ of ${\mathbb R}^n$. More precisely:

\medskip
$(*)$ if we know any cube $Q\in T_k$, then all cubes in $T_k$
are given by $x(Q)+\ell_Qu+\ell_QQ_0$, where $u\in{\mathbb Z}^n$ and $Q_0=[0,1)^n$.
\medskip

The tilings $T_k$ and $T_{k+1}$ agree with each other in the sense that every cube in $T_{k+1}$ is a child of a unique cube in $T_k$, so
$T_{k+1}$ is a refinement of $T_k$. Property (DL-3) can be restated in terms of tilings $T_k$ in literally the same way:
every compact $K\subset {\mathbb R}^n$ is contained in some cube in one of the tilings $T_k$.

This all makes it tempting to look at dyadic lattices from another point of view and to introduce the following definition.

\begin{definition}\label{multires}
Let $T=\cup_{k\in\mathbb Z}T_k$ be a system of cubes comprised of regular (in the sense of $(*)$) tilings $T_k$ of ${\mathbb R}^n$.
We say that $T$ is a dyadic multiresolution of ${\mathbb R}^n$ if
\begin{enumerate}[{(DM-1)}]
\item
for every $k$, the tiling $T_{k+1}$ consists exactly of dyadic children of cubes in $T_k$;
\item
for every compact $K\subset {\mathbb R}^n$, there exists some $k\in{\mathbb Z}$ and $Q\in T_k$ such that $Q\supset K$.
\end{enumerate}
\end{definition}

\begin{remark}\label{conddm}
Condition (DM-1) may be quite unpleasant to check directly. However, there is a nice simple criterion for it to hold: if we can find a cube
$Q=x(Q)+\ell_Q Q_0\in T_k$ such that its ``corner child" $x(Q)+\frac{\ell_Q}{2}Q_0$ is in $T_{k+1}$, then (DM-1) holds.

Indeed, then every cube $Q'\in T_{k+1}$ can be represented as
$$x(Q)+\frac{\ell_Q}{2}u+\frac{\ell_Q}{2}Q_0,\quad u\in {\mathbb Z}^n.$$
Writing $u=2v+\e, v\in {\mathbb Z}^n,\e\in \{0,1\}^n$ and observing that the dyadic children of a cube $R$ are $x(R)+\frac{\ell_R}{2}\e+\frac{\ell_R}{2}Q_0,\e\in\{0,1\}^n,$
we see that $Q'$ is a child of $x(Q)+\ell_Qv+\ell_QQ_0\in T_k$.
\end{remark}

Note that for all $k$, if the tiling $T_{k+1}$ consists of dyadic children of cubes in $T_k$, then $T_{k+2}$ consists of dyadic grandchildren of cubes in $T_k$, etc.,
so $T_{k+m}=\cup_{Q\in T_k}{\mathcal D}_m(Q)$ where $\mathcal D_m(Q)$ is the $m$-th generation of $\mathcal D(Q)$.
Since the cubes in $T_k$ are pairwise disjoint, this implies, in particular, that for two cubes $Q\in T_k$
and $Q'\in T_{k+m}$ in a multiresolution, the conditions $Q'\subset Q$ and $Q'\in {\mathcal D}(Q)$ are equivalent.

Now we are ready to show that every dyadic multiresolution is a dyadic lattice. Indeed, if $Q\in T_k$ for some $k$, then every child of $Q$ belongs to $T_{k+1}$,
establishing (DL-1). Property (DL-3) is the same as property (DM-2).
Finally, if $Q',Q''\in T$, then, by (DM-2), there exists $Q\in T$ such that $Q',Q''\subset Q$. But, as we have seen above, this is equivalent to $Q',Q''\in {\mathcal D}(Q)$,
establishing (DL-2).

\subsection{The baby Lebesgue differentiation theorem}

Dyadic lattices are often used in analysis to discretize some statements and arguments about functions
on $\mathbb R^n$. The general idea is that every reasonable function $f:\mathbb R^n\to\mathbb R$ is
constant up to a negligible error on sufficiently small cubes $Q\in\mathscr D$. This assertion can be understood
literally if $f$ is continuous. If the function $f$ is merely measurable, this general principle should
be interpreted with some caution. The particular formulation we present in this section is
not the strongest one but it is easy to prove and suffices for all purposes of this book.

Let $f$ be a measurable function that is finite almost everywhere. We say that $x\in {\mathbb R}^n$ is a (dyadic) weak
Lebesque point of $f$ if for every $\e,\la>0$, we have
$$|\{y\in Q:|f(y)-f(x)|>\e\}|<\la|Q|$$
for at least one cube $Q\in {\mathscr D}$ containing $x$.

\begin{theorem}
Almost every point $x\in {\mathbb R}^n$ is a weak Lebesgue point of $f$.
\end{theorem}

\begin{proof}
Fix $\e,\la>0$ and consider the set of ``bad'' points $x\in {\mathbb R}^n$ for which for every cube $Q\in {\mathscr D}$
containing $x$, there is a set $U(Q,x)\subset Q$ of measure $|U(Q,x)|\ge \la |Q|$ such that $|f-f(x)|\ge \e$
on $U(Q,x)$. For $a\in {\mathbb R}, R>0$, let
$$E(a,R)=\{x:|x|\le R, |f(x)-a|<\e/3\}.$$
Note that $\cup_{l\in\mathbb Z,m\ge 1}E(l\e/3,m)={\mathbb R}^n$, so it is enough to show that the outer Lebesgue measure
$\mu$ of the set $ E_{\text{bad}}(a,R)$ of the bad points in each $E(a,R)$ is~$0$. Assume $\mu>0$. By the definition of the outer Lebesgue measure,
we can choose a dyadic cover%
\footnote{The standard definition of the Lebesgue outer measure uses arbitrary parallelepipeds, but each
parallelepiped can be covered by finitely many dyadic cubes of the total volume as close to the volume
of the parallelepiped as one wants.}
$E_{\text{bad}}(a,R)\subset \cup_{Q\in\s}Q$ $(\s\subset {\mathscr D})$ with $\sum_{Q\in \s}|Q|<(1+\la)\mu$.
We can also remove from $\s$ all cubes $Q$ disjoint with $E_{\text{bad}}(a,R)$.
But then for each cube $Q\in \s$, we can choose a point $x(Q)\in Q\cap E_{\text{bad}}(a,R)$ and
note that 
the corresponding  set $U(Q,x(Q))$ is disjoint with $E(a,R)$ because all values of $f$ on $U(Q,x(Q))$ differ by at least $\e$
from $f(x(Q))\in (a-\e/3,a+\e/3)$.  Thus,
$$\mu\le \sum_{Q\in \s}|Q\setminus U(Q,x(Q))|\le (1-\la)\sum_{Q\in \s}|Q|\le (1-\la^2)\mu,$$
which is a clear contradiction.
\end{proof}

\section{The Three Lattice Theorem}
Quite often we can easily estimate some quantities in terms of averages $f_Q=\frac{1}{|Q|}\int_Qf$ of positive functions over
some cubes in ${\mathbb R}^n$ but not necessarily over dyadic cubes.

We would like to estimate these averages by the averages of the same kind but taken over the cubes in some dyadic lattice $\md$ only.
To this end, note that if $Q_1\subset Q\subset Q_2$, then
$$\frac{|Q_1|}{|Q|}f_{Q_1}\le f_Q\le \frac{|Q_2|}{|Q|}f_{Q_2}.$$

Thus, our task can be accomplished with decent precision if for an arbitrary cube $Q$, we can find a pair of dyadic cubes $Q_1,Q_2\in\md$
such that $Q_1\subset Q\subset Q_2$ and the volume ratios $\frac{|Q_1|}{|Q|}$ and $\frac{|Q_2|}{|Q|}$ are not too small or too large
respectively. Finding $Q_1$ is never a problem: just take the cube in $\md$ containing $c(Q)$ whose sidelength is between $\ell_Q/4$
and $\ell_Q/2$. However, it may easily happen that we need to go quite high up in the dyadic lattice to meet the first single cube covering $Q$ (see Figure \ref{badcaseup}).

\begin{figure}[ht]
\begin{center}
\includegraphics[width=5cm,height=5cm]{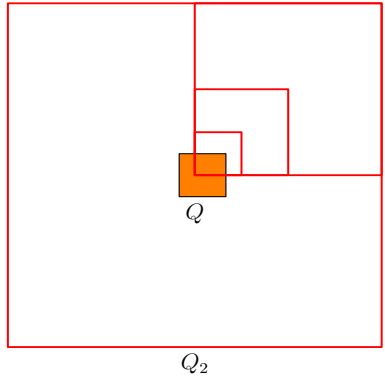}
\caption[]{The smallest cube $Q_2\in \md$ containing $Q$ can be arbitrarily many times larger than $Q$.}
\label{badcaseup}
\end{center}
\end{figure}

We see that we can cover $Q$ by the union of few adjacent cubes of comparable size though. More precisely, if we take the dyadic cube $Q_2$
containing $x(Q)$ and such that $\ell_Q/2<\ell_{Q_2}\le \ell_Q$, then
$$Q\subset \widetilde Q_2=x(Q_2)+3(Q_2-x(Q_2))$$
(this cube with the same corner as $Q_2$ but of triple size is comprised of $3^n$ lattice cubes). Unfortunately, the set
$\widetilde\md=\{\widetilde Q:Q\in\md\}$ is not a dyadic lattice. Or is it? No, of course not: the cubes in $\widetilde \md$ overlap in fancy
ways, which is not allowed for cubes in one lattice. In {\em one} lattice? And now the next logical step is inevitable: $\widetilde \md$ is not a {\em single}
lattice, indeed, but it is a union of {\em several} ($3^n$) dyadic lattices.

\begin{theorem}\label{three}{\rm{(The Three Lattice Theorem)}}
For every dyadic lattice $\md$, there exist $3^n$ dyadic lattices $\md^{(1)},\dots,\md^{(3^n)}$ such that
$$\widetilde \md=\{x(Q)+3(Q-x(Q)):Q\in\md\}=\bigcup_{j=1}^{3^n}\md^{(j)}$$
and for every cube $Q\in \md$ and $j=1,\dots,3^n$, there exists a unique cube $\widetilde R\in \md^{(j)}$ of
sidelength $\ell_{\widetilde R}=3\ell_Q$ containing $Q$.
\end{theorem}

\begin{remark}\label{remcor}
An immediate corollary to the  Three Lattice Theorem is that for every $m=1,2,\dots$, there exist $3^{mn}$ dyadic
lattices $\md_{(1)},\dots,\md_{(3^{mn})}$ such that for each set of arbitrary cubes $Q_1,\dots,Q_m$, one can find
$i\in\{1,\dots,3^{mn}\}$ and cubes $\wt Q_1,\dots,\wt Q_m \in \md_{(i)}$ so that for each $k=1,\dots,m$,
one has $\wt Q_k\supset Q_k$ and $|\wt Q_k|\le 3^{mn}|Q_k|$.

This can be easily proved by induction on $m$. The base case $m=1$ immediately follows from the discussion before the
statement of the theorem.
Just take any dyadic lattice $\md$ and consider the lattices $\md^{(j)}$ given by the first assertion of the theorem.
Note that the second assertion has not been used for this case though it will be crucial for the induction step.

Assume now that the claim holds for some $m$ and
 $\md_{(1)},\dots,\md_{(3^{mn})}$ are some dyadic lattices satisfying the required property. Applying the
Three Lattice Theorem to each of these lattices, we get a new set of $3^{(m+1)n}$ lattices $\md_{(i)}^{(j)}$ so that for each $i$,
the set
$$\wt\md_{(i)}=\{x(Q)+3(Q-x(Q)):Q\in\md_{(i)}\}$$
is the union of  $\md_{(i)}^{(j)}$, $j=1,\dots,3^n$.
To finish the induction step, it will suffice to show that this new set of lattices satisfies the required property for $m+1$.

Take any cubes $Q_1,\dots,Q_{m+1}$. By the induction assumption, we can find $i$ and some cubes
$\wt Q_k'\in \md_{(i)}$ ($k=1,\dots,m$) so that $\wt Q'_k\supset Q_k$ and $|\wt Q'_k|\le 3^{mn}|Q_k|$.
This leaves us with the last cube $Q_{m+1}$ to take care of. As it was observed earlier, we can find $j$ and a cube
$\wt Q_{m+1}\in \md_{(i)}^{(j)}$ so that $ \wt Q_{m+1}\supset Q_{m+1}$ and
$|\wt Q_{m+1}|\le 3^n|Q_{m+1}|\le 3^{(m+1)n}|Q_{m+1}|$. It remains to note that, by the second assertion in the
Three Lattice Theorem, we can extend each cube $\wt Q_k'\in \md_{(i)}$ to a $3$ times larger cube
$\wt Q_k\in \md_{(i)}^{(j)}$.

\end{remark}

\begin{proof}[Proof of Theorem \ref{three}]
We start with a picture in ${\mathbb R}^1$:

\begin{figure}[ht]
\begin{center}
\includegraphics[width=10cm,height=5cm]{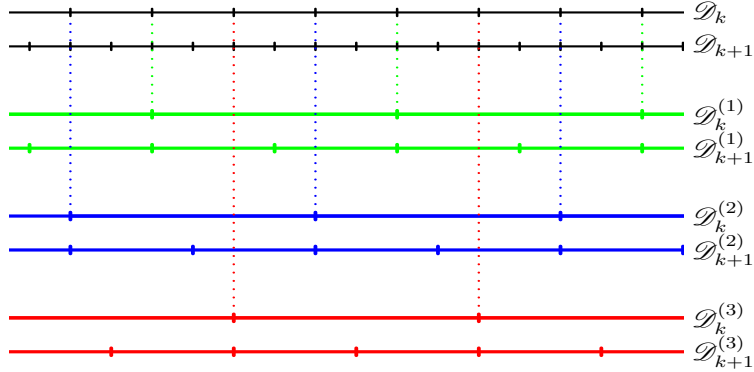}
\caption[]{The lattice $\md$ and three lattices $\md^{(j)}$ shown at two consecutive generations.}
\label{threelattices}
\end{center}
\end{figure}

For the formal argument, it will be convenient to assume that $Q_0=[0,1)^n\in \md_0$ (which can always be achieved by translation and scaling).
Then every other cube $Q\in\md_k$ can be represented as $x+2^{-k}Q_0$, where $x$ is a dyadic rational vector; more precisely,
$2^mx\in {\mathbb Z}^n$
for $m\ge\max(k,0)$. Note that the ``corner child" of $Q=x+2^{-k}Q_0\in\md_k$ is $x+2^{-(k+1)}Q_0\in~ \md_{k+1}$ and that all other cubes in ${\md}_k$ are given by $x+2^{-k}u+2^{-k}Q_0\, (u\in {\mathbb Z}^n)$.

Consider the generation $\widetilde \md_0=\{x+3Q_0:x\in {\mathbb Z}^n\}$ of triple cubes corresponding
to cubes in $\md_0$. Any cube $\widetilde Q=x+3Q_0$ can (and must) be put
into the same dyadic lattice $\md^{(j)}$ with all cubes $\widetilde Q'=x'+3Q_0$ for which $x-x'\in (3{\mathbb Z})^n$.
This gives us an easy and natural way to split one generation $\widetilde \md_0$ into $3^n$ tilings.
Introduce the equivalence relation on $\mathbb Z^n$ by $x\sim y\Leftrightarrow x-y\in (3{\mathbb Z})^n$. The cubes $x+3Q_0,y+3Q_0\in \widetilde\md_0$ are put in the same tiling
if and only if $x\sim y$.

Since $2$ and $3$ are coprime, the divisibility of an integer by $3$ is not affected by multiplication or division by $2$,
so we can easily extend this equivalence relation to the entire set of corners of all dyadic cubes in $\md$
saying that the corners $x'$ and $x''$ of two cubes in $\md$
are equivalent if $2^m(x'-x'')\in (3{\mathbb Z})^n$ for sufficiently large $m$. Since our equivalence relation is invariant under dyadic scaling, we see that the cubes $\widetilde Q\in \widetilde\md_k$ with corners from one equivalence class tile ${\mathbb R}^n$ for each $k\in\mathbb Z$ the same way as it was for $k=0$. Moreover, each cube $Q\in \md_k$ is contained in some
cube of this tiling.

\begin{figure}[ht]
\begin{center}
\includegraphics[width=5cm,height=5cm]{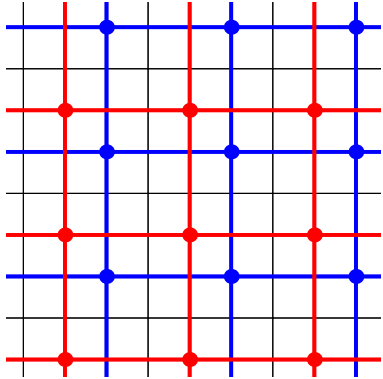}
\caption[]{The tilings of $\mathbb R^2$ by triple cubes with corners in two different equivalence
classes (blue and red).}
\label{twotilings}
\end{center}
\end{figure}

We have already seen that there are $3^n$ equivalence classes in each generation. Note now that if we have $3^n+1$ dyadic rational vectors $x_j$
and take $m$ so large that $2^mx_j\in {\mathbb Z}^n$ for all $j$, then by the pigeonhole principle, the difference of some two of the integer vectors
$2^mx_j$ is in $(3{\mathbb Z})^n$, so there cannot be more than $3^n$ equivalence classes total. Thus

\medskip

(a) there are exactly $3^n$ equivalence classes ${\mathcal E}_j$;

(b) for each $k$, the cubes $\widetilde Q\in\widetilde\md_k$ with corners in the same equivalence class ${\mathcal E}_j$ tile ${\mathbb R}^n$ in a regular way
and the original tiling of ${\mathbb R}^n$ by the cubes $Q\in{\md_k}$ is a refinement of this tiling.

\medskip

Denote
$$\md^{(j)}=\{\widetilde Q\in\widetilde\md:x(\widetilde Q)\in {\mathcal E}_j\}.$$
We will show that each $\md^{(j)}$ is a dyadic multiresolution.
Property (b) already establishes that each $\md^{(j)}$ satisfies (DM-2): just take any compact $K\subset {\mathbb R}^n$, find a cube $Q$ in some $\md_k$ with $Q\supset K$, and then take the cube $\widetilde Q\in \widetilde\md_k$ containing $Q$ with $x(\widetilde Q)\in {\mathcal E}_j$.

To show (DM-1),
take any cube  $\widetilde Q=x(Q)+3(Q-x(Q))\in \md_k^{(j)}$. Then $x(\widetilde Q)=x(Q)\in {\mathcal E}_j$. The corner child $\widetilde Q'$
of $\widetilde Q$ is just $x(Q)+3(Q'-x(Q))$ where $Q'$ is the corner child of $Q$. Since $\widetilde Q'$ and $\widetilde Q$ share the common corner,
it follows that $\widetilde Q'\in\md_{k+1}^{(j)}$ and, by Remark \ref{conddm}, that $\md_{k+1}^{(j)}$ consists exactly of the dyadic children of
the cubes in $\md_{k}^{(j)}$.

Finally, the last statement in property (b) is equivalent to the last assertion of the theorem.
\end{proof}

We want to emphasize that the key non-trivial part of the argument above is hidden not in the verification of some
particular property: each of those is fairly straightforward.  Neither is it hidden in some ingenious step or  in
 combining the intermediate statements in some fancy ways,
though the order of steps does matter if one wants to run the argument in a smooth way and without unnecessary repetitions. The crucial point
is the very possibility to extend the equivalence relation that naturally arises when looking at one generation to all other generations and,
most importantly, {\em across} generations in one wide sweep. It is this extension across generations, which requires the coprimality of the
numbers $2$ and $3$ to work and which is ultimately responsible for making it possible to combine the individual tilings
coming from different generations into finitely many complete dyadic multiresolutions (the reader is encouraged to try to draw a picture similar to Figure \ref{threelattices} to see what will go
wrong if some even number is used in place of $3$).

\section{The forest structure on a subset of a dyadic lattice}

Let ${\mathscr D}$ be a dyadic lattice. Let ${\mathcal S}\subset {\mathscr
D}$
be some family of dyadic cubes.
We will view ${\mathcal S}$ as the set of vertices of a graph
$\Gamma_{\mathcal S}$. We will join two cubes $Q,Q'\in {\mathcal S}$ by a
graph edge
if $Q'\subset Q$ and there is no intermediate cube $Q''\in\s$ (i.e., a cube
such that
$Q'\subsetneqq Q''\subsetneqq Q$) between $Q$ and $Q'$.

\begin{figure}[ht]
\begin{center}
\includegraphics[width=8cm,height=6cm]{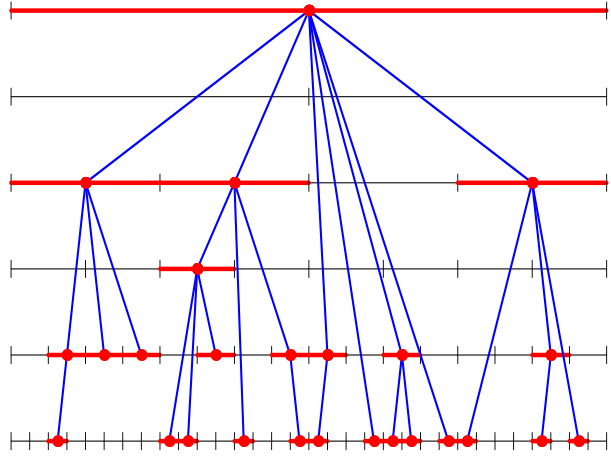}
\caption[]{The graph $\Gamma_\s$ (red intervals are in $\s$).}
\label{gammasedges}
\end{center}
\end{figure}

Though we prefer to view $\G_{\s}$ as a non-oriented graph, there are two
natural directions of motion on $\G_{\s}$: up (from smaller cubes to larger
ones)
and down (from larger cubes to smaller ones).

Note that the graph $\Gamma_{\mathcal S}$ cannot contain any cycle because
the ``lowest'' (smallest) cube $Q$ in that cycle would have to connect to two
different cubes $Q'$ and $Q''$ of larger size.
Then $Q'\cap Q''\supset Q\not=\varnothing$, so either $Q'\supset Q''$, or
$Q''\supset Q'$, but in each case one of the cubes cannot be connected
to $Q$ because another one is on the way.

As usual, we can define the graph distance $d_{\s}(Q',Q'')$ between two
cubes $Q',Q''\in\s$ as the number of graph edges
in the shortest path from $Q'$ to $Q''$ (if there is no such path, we put
$d(Q',Q'')=+\infty$).

Note that if $Q'\supset Q''$, then we can ascend from $Q''$ to $Q'$ in the
full dyadic lattice $\md$ in finitely many steps, each of which goes from
a cube to its parent. Since the total ascent is finite, we can meet only
finitely many cubes in $\s$ on the way. Moreover, any path from
$Q''$ to $Q'$ in $\G_{\md}$ goes through every cube in this ascent.
These three simple observations have several useful implications, which we will
now state for future reference.

\medskip

1) If $Q',Q''\in\s$ and $Q'\supsetneqq Q''$, then $d_{\s}(Q',Q'')$ is finite
and equals $1$ plus the number of intermediate cubes in $\s$ between $Q'$
and $Q''$.

\medskip

2) Let $R\in\md$. Assume that there is at least one cube in $\s$ containing $R$.
Then we shall call the smallest
dyadic cube in $\s$ containing $R$ the
$\s$-roof of $R$ and denote it by $\widehat R_{\s}$.
Note that $\widehat R_{\s}$ is also the first cube in $\s$
we meet on the upward path in $\Gamma_\md$ starting from $R$.

The cubes $R\in\md$ that have no $\s$-roof are contained in no cube from
$\s$.
All other cubes naturally split into houses
$$H_\s(Q)=\{R\in \md:\widehat R_{\s}=Q\}$$
of cubes located under the same roof $Q\in \s$ (see Figure \ref{houses}). The cubes in $H_\s(Q)$ are the
cubes one can reach on the way down from $Q$ in $\Gamma_\md$ before meeting another cube in
$\s$.
Note that each cube $R\in\s$ is its own roof and the next cube in $\s$ on
the way up from $R$ (the cube connected to $R$ by an edge in $\G_{\s}$)
is not the $\s$-roof of $R$, but the $\s$-roof of the parent of $R$.

\begin{figure}[ht]
\begin{center}
\includegraphics[width=8cm,height=6cm]{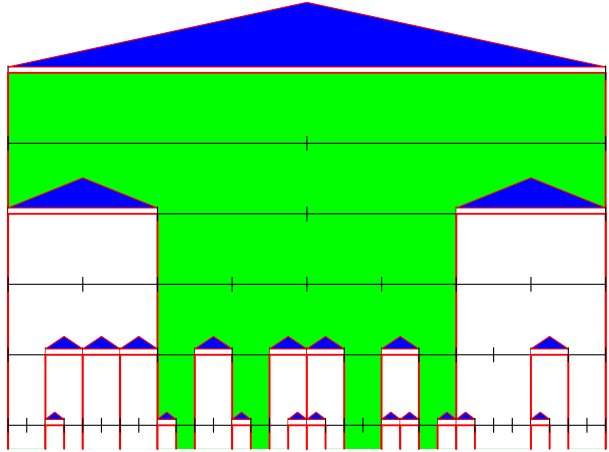}
\caption[]{The roofs $Q\in\s$ (red intervals under blue triangles) and the houses $H_\s(Q)$ (with red walls). The house
of the top interval is highlighted in green.}
\label{houses}
\end{center}
\end{figure}

3) If for every compact set $K\subset {\mathbb R}^n$ there exists a cube
$Q\in\s$ containing $K$ (we will call such families regular), then $\G_{\s}$
is connected and $\md=\cup_{Q\in\s}H_\s(Q)$.

\medskip

4) If $Q'\supset Q''$ and $Q',Q''\in\s'\subset \s$, then
$d_{\s'}(Q',Q'')\le d_{\s}(Q',Q'')$.
This is also true for the case when the cubes are disjoint, provided that
$d_{\s'}(Q',Q'')<+\infty$, but we will never need this fact, so we leave
its proof
to the reader.

\medskip

5) If $Q,Q'\in\s$, $R\in H_\s(Q)$, and $R\supset Q'$, then $d_{\md}(R,Q')\ge
d_{\s}(Q,Q')$. This follows from the fact that to reach $R$ from $Q'$
one needs to pass through all intermediate cubes $Q''\in\s$ on the way up.

\medskip

6) If $Q\in{\mathcal S}$ and $Q_1,\dots,Q_m\in {\mathcal S}$ are some subcubes of $Q$
with $d_{\s}(Q,Q_1)=\dots=d_{\s}(Q,Q_m)$, then the cubes $Q_1,\dots,Q_m$ are pairwise disjoint.

\section{Stopping times and augmentation}
\subsection{Stopping times}
Suppose that we have some condition (a boolean function ${\mathcal P}:\{(Q,Q')$
$\in {\mathscr D}\times {\mathscr D}:Q\supset Q'\}\to \{\text{true},\text{false}\}$) that may hold or not hold for every pair of dyadic cubes $Q\supsetneqq Q'$, but such that ${\mathcal P}(Q,Q)$ is always true.

Take any cube $Q\subset {\mathscr D}$. We say that a cube $Q'\subset Q$ can be reached from $Q$ in one step if ${\mathcal P}(Q,Q')$ fails
but ${\mathcal P}(Q,Q'')$ holds for all $Q''$ with $Q'\subsetneqq Q''\subset Q$. We say that a cube $Q'\subset Q$ can be reached from $Q$
in $m$ steps if there is a chain $Q=Q_0\supset Q_1\supset\dots\supset Q_m=Q'$ in which $Q_{j+1}$ can be reached in one step from $Q_j\,
(j=0,1,\dots,m-1)$.
The system of all cubes that can be reached from $Q$ in finitely many steps (including $Q$, which can be reached from itself in $0$ steps) is called the system of stopping cubes associated with the initial cube
$Q$ and condition ${\mathcal P}$ and denoted $\text{stop}(Q,{\mathcal P})$ or, merely, $\text{stop}(Q)$ if the condition is clear from the context.

\begin{figure}[ht]
\begin{center}
\includegraphics[width=10cm,height=5cm]{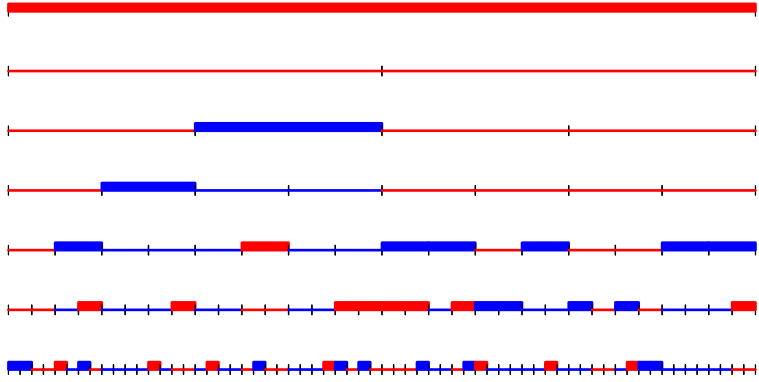}
\caption[]{The system $\text{stop}(Q,\mathcal P)$ (thick intervals) for the condition ``$Q$ and $Q'$ are of the same color''.}
\label{stopQP}
\end{center}
\end{figure}

The graph $\Gamma_{\text{stop}(Q,{\mathcal P})}$ is a tree with the top cube $Q$.  The cubes $Q'$ that can be reached from $Q$ in $m$ steps
are exactly those with $d_{\text{stop}(Q,{\mathcal P})}(Q,Q')=m$.

The most important feature of the system $\text{stop}(Q,{\mathcal P})$ is that
for every dyadic subcube  $R\subset Q$, its roof $\widehat R_{\text{stop}(Q,{\mathcal P})}$ is well-defined and satisfies
${\mathcal P}(\widehat R_{\text{stop}(Q,{\mathcal P})},R)=\text{true}$. Indeed, since $Q\in\text{stop}(Q)$, and
$R\subset Q$, the $\text{stop}(Q)$-roof of $R$ exists. On the other hand, if we descend from $\widehat R_{\text{stop}(Q,{\mathcal P})}$ to $R$, the condition ${\mathcal P}(\widehat R_{\text{stop}(Q,{\mathcal P})},R')$ should hold for all cubes $R'$ on the way (including $R$). Otherwise the first cube $R'$ on that way down for which it fails can be reached from $\widehat R_{\text{stop}(Q,{\mathcal P})}$ in one step, so it also belongs to $\text{stop}(Q,{\mathcal P})$ and contains $R$. By the definition of the roof, we then must
have $R'=\widehat R_{\text{stop}(Q,{\mathcal P})}$, but this is impossible because
$\mathcal P( \widehat R_{\text{stop}(Q,{\mathcal P})},\widehat R_{\text{stop}(Q,{\mathcal P})})=\text{true}$.

Thus, the whole set $\mathcal D(Q)$ of subcubes of $Q$ is a disjoint union of the houses $H_{\text{stop}(Q,{\mathcal P})}(Q')$,
$Q'\in \text{stop}(Q,{\mathcal P})$ and for every $Q'\in\text{stop}(Q,{\mathcal P})$ and $R \in H_{\text{stop}(Q,{\mathcal P})}(Q')$, the
condition $\mathcal P(Q',R)$ holds.

Another useful property is self-similarity: if $Q'\in \text{stop}(Q,{\mathcal P})$, then the family of all subcubes of $Q'$
in $ \text{stop}(Q,{\mathcal P})$ is exactly $ \text{stop}(Q',{\mathcal P})$.

\subsection{Augmentation}
Suppose that $\s$ is any subset of ${\mathscr D}$ and that with every cube $Q\in \s$ some family ${\mathcal F}(Q)$ of dyadic subcubes of $Q$ is associated
so that $Q\in {\mathcal F}(Q)$.
The augmented family $\widetilde\s$ is then the union $\cup_{Q\in\s}\widetilde{\mathcal F}(Q)$ where $\widetilde{\mathcal F}(Q)$ consists of all cubes
$Q'\in {\mathcal F}(Q)$ that are not contained in any $R\in \s$ with $R\subsetneqq Q$ (see Figure \ref{augmented}).
The augmented family $\widetilde \s$ always contains the original family $\s$ but, in general, can be much larger.

\begin{figure}[ht]
\begin{center}
\includegraphics[width=10cm,height=5cm]{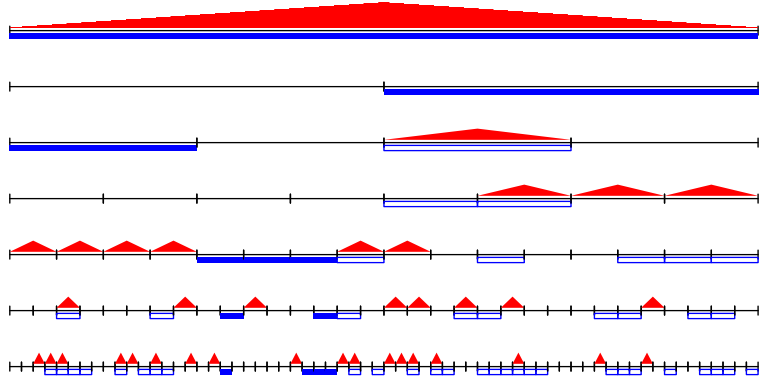}
\caption[]{The families $\mathcal F(Q) $ (all blue rectangles) and $\wt{\mathcal F}(Q)$ (solid blue rectangles)
for an interval $Q$ (the top interval) in $\s$ (red triangles).}
\label{augmented}
\end{center}
\end{figure}

An important case of augmentation is when ${\mathcal F}(Q)=\text{stop}(Q,{\mathcal P})$ for some condition ${\mathcal P}$.
In this case, the augmented system $\widetilde\s$ contains $\s$ and enjoys the property that ${\mathcal P}(\widehat R_{\widetilde\s},R)$ holds
for every $R\in {\mathscr D}$. Indeed, $\widehat R_{\widetilde\s}\in \text{stop}(Q,P)$ for some $Q\in\s$. Then $\widehat R_{\widetilde\s}=
\widehat R_{\text{stop}(Q,P)}$ because the only way the two could be different is that $\widehat R_{\text{stop}(Q,P)}$ would be removed as a cube
contained in some $Q'\subsetneqq Q$ from $\s$, but in that case $Q'$ would block all ways up from $R$ to $\widetilde{\mathcal F}(Q)$ as well.

\section{Sparse and Carleson families}
Let ${\mathscr D}$ be a dyadic lattice.
\begin{definition}\label{sparse2}
Let $0<\eta<1$. A collection of cubes ${\mathcal S}\subset {\mathscr D}$ is called $\eta$-sparse if
one can choose pairwise disjoint measurable sets $E_Q\subset Q$ with $|E_Q|\ge \eta|Q|$ ($Q\in {\mathcal S}$).
\end{definition}

\begin{definition}\label{carl}
Let $\La>1$. A family of cubes ${\mathcal S}\subset {\mathscr D}$ is called $\La$-Carleson if for every cube $Q\in {\mathscr D} $
we have
$$\sum_{P\in {\mathcal  S}, P\subset Q}|P|\le \La|Q|.$$
\end{definition}

\subsection{The equivalence of the Carleson and sparseness conditions}
It is almost obvious that every $\eta$-sparse family is $\frac{1}{\eta}$-Carleson.
Namely, we can just write
$$
\sum_{P\in {\mathcal  S}, P\subset Q}|P|\le
\eta^{-1}\sum_{P\in {\mathcal  S}, P\subset Q}|E_P|
\le \eta^{-1}|Q|\,.
$$
That every $\La$-Carleson family is $\frac{1}{\La}$-sparse is, however, much less obvious
(which was one of the reasons for the duplication of terminology in the first place), so our first goal will be to show exactly that.

\begin{lemma}
If ${\mathcal S}$ is $\La$-Carleson, then ${\mathcal S}$ is $\frac{1}{\La}$-sparse.
\end{lemma}

\begin{proof}
It would be easy to construct the sets $E_Q$ if the lattice had a bottom layer ${\mathscr D}_K$. Then we would start with considering all cubes $Q\in {\mathcal S}\cap {\mathscr D}_K$
and choose any sets $E_Q\subset Q$ of measure $\frac{1}{\La}|Q|$ for them. After that we would just go up layer by layer and for each cube $Q\in {\mathcal S}\cap  {\mathscr D}_k,k\le K,$
choose a subset $E_Q$ of measure $\frac{1}{\La}|Q|$ in $Q\setminus \cup_{R\in\s,R\subsetneqq Q}E_R$. Note that for every
$Q\in\mathcal S$, we have
$$\Big|\bigcup_{R\in {\mathcal S},R\subsetneqq Q}E_R\Big|\le \frac{1}{\La}\sum_{R\in {\mathcal S},R\subsetneqq Q}|R|\le \frac{\La-1}{\La}|Q|=\Big(1-\frac{1}{\La}\Big)|Q|,$$
so such choice is always possible.

The problem with the absence of a bottom layer is that going down in a similarly simple way is not feasible. A natural idea is to run the above construction for each
$K=0,1,2,\dots$ and then pass to the limit, but we need to make sure that the resulting subsets do not jump around wildly.
Fortunately, it is not that hard.  All we have to do is to replace ``free choice'' with ``canonical choice''.

Fix $K\ge 0$. For $Q\in {\mathcal S}\cap (\cup_{k\le K}{\mathscr D}_k)$ define the sets $\widehat E_Q^{(K)}$ inductively as follows.
 If $Q\in {\mathcal S}\cap {\mathscr D}_K$, set $\widehat E_Q^{(K)}=x(Q)+{\La}^{-1/n}(Q-x(Q))$
(the cube with the same corner as $Q$ of measure $|\widehat E_Q^{(K)}|=\frac{1}{\La}|Q|$). For  $Q\in {\mathcal S}\cap {\mathscr D}_k, k<K$, the sets $\widehat E_Q^{(K)}$
will not necessarily be cubes. Namely, if $\widehat E^{(K)}_R$ are already defined for $R\in  {\mathcal S}\cap (\cup_{k+1\le i\le K}{\mathscr D}_i)$, then for $Q\in
 {\mathcal S}\cap {\mathscr D}_k$, set
$$\widehat E^{(K)}_Q=F^{(K)}_Q\cup\big(x(Q)+t(Q-x(Q))\big),$$
where
$$
F^{(K)}_Q=\bigcup_{R\in {\mathcal S},R\subsetneqq Q}\widehat E^{(K)}_R$$
and $t\in [0,1]$ is the largest number for which the set
$$E^{(K)}_Q=[x(Q)+t(Q-x(Q))]\setminus F^{(K)}_Q$$
satisfies $|E_Q^{(K)}|\le \frac{1}{\Lambda}|Q|$ (see Figure \ref{sand}).

\begin{figure}[ht]
\begin{center}
\includegraphics[width=6cm,height=6cm]{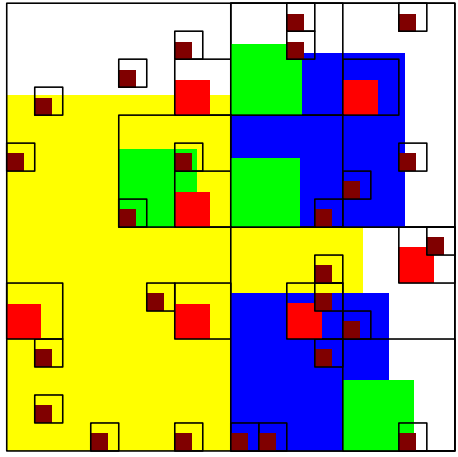}
\caption[]{The construction from the bottom level (brown) to $4$ levels up (yellow). For the largest cube $Q\in \s$ shown, the set
$\widehat E^{(K)}_Q$ is the total colored area, the set $E^{(K)}_Q$ is the yellow area, and the set $F^{(K)}_Q$ is the  area
colored with colors other than yellow.}
\label{sand}
\end{center}
\end{figure}

The above considerations (together with the continuity and monotonicity of the function
$t\mapsto |[x(Q)+t(Q-x(Q))]\setminus F^{(K)}_Q|$) imply that such $t$ exists
and we have $|E_Q^{(K)}|=\frac{1}{\Lambda}|Q|$.

We claim that $\widehat E^{(K)}_Q\subset\widehat  E^{(K+1)}_Q$  for every $Q\in {\mathcal S}\cap (\cup_{k\le K}{\mathscr D}_k)$. Indeed, if $Q\in\md_K$, then $\widehat E^{(K)}_Q$ is just the cube with corner $x(Q)$ of volume $\frac 1\Lambda|Q|$. On the other hand, $\widehat E^{(K+1)}_Q$ contains a cube with the same corner for which the volume of the difference between that cube and
some other set is as large. It remains to note that out of two cubes with the same corner the cube of larger volume contains
the cube of smaller volume.

From here we can proceed by backward induction. Assume that $\widehat E^{(K)}_Q\subset\widehat  E^{(K+1)}_Q$ for every $Q\in {\mathcal S}\cap (\cup_{k<i\le K}{\mathscr D}_i)$. Take any $Q\in \mathcal S\cap \md_k$. Then, the induction assumption
implies that $ F^{(K)}_Q\subset F^{(K+1)}_Q$. Let $Q'=x(Q)+t'(Q-x(Q))$ be the cube added to $ F^{(K)}_Q$ when constructing
$ \widehat E^{(K)}_Q$. Then
$$
|Q'\setminus  F^{(K+1)}_Q|\le |Q'\setminus F^{(K)}_Q|=\frac 1\Lambda|Q|
$$
so the value of $t$ that is chosen in the construction of $\widehat E^{(K+1)}_Q$ is not less than $t'$. Thus, the cube that we add
to $F^{(K+1)}_Q$ when building $\widehat E^{(K+1)}_Q$ contains the cube that we add
to $F^{(K)}_Q$ when building $\widehat E^{(K)}_Q$, and the claim follows.


Now, for $Q\in {\mathcal S}\cap \md_k$, define
$$
\widehat E_Q=\lim_{K\to \infty}\widehat E^{(K)}_Q=\bigcup_{K=k}^{\infty}\widehat E^{(K)}_Q\subset Q\,.
$$
For each $K$, we have
$$|E^{(K)}_Q|=|\widehat E^{(K)}_Q\setminus F^{(K)}_Q|=\frac{1}{\Lambda}|Q|\,.$$
Note now that the sets $F^{(K)}_Q$ also form an increasing (in $K$) sequence, so for each $Q\in\s$,  the limit set
$$E_Q=\lim_{K\to\infty}E_Q^{(K)}=\widehat E_Q\setminus\Bigl(\lim_{K\to \infty} F^{(K)}_Q\Bigr)
=\widehat E_Q\setminus\Bigl(\bigcup_{R\in\s,R\subsetneqq Q} \widehat E_R\Bigr)
$$
exists, is contained in $Q$ and has the required measure. Finally, $E_Q$ are, obviously, disjoint.
\end{proof}

The sparseness property is something that can be readily used when working with systems of cubes that are already known to be
sparse while the Carleson property is something that can be easily verified in many cases where the sparseness condition is not
obvious at all. For example, it is clear straight from the definition that the union of $N$ Carleson systems with constants $\Lambda_1,\dots,\Lambda_N$
is a Carleson system with constant $\Lambda_1+\dots+\Lambda_N$, while to see directly that the union of $\eta_j$-sparse systems is $(\sum_{j=1}^N\eta_j^{-1})^{-1}$-sparse is next to impossible.

\subsection{Anti-Carleson stack criterion}
We will now present a simple criterion for a system of cubes to be Carleson.
\begin{definition}
Let $Q\in {\mathscr D}$ and $\eta\in (0,1)$. An $\eta$-anti-Carleson stack of height $M$ with the top cube $Q$ is a finite family ${\mathcal F}\subset {\mathcal D}(Q)$
containing $Q$ and such that
$$\sum_{Q'\in {\mathcal F}, d_{\mathcal F}(Q',Q)=M}|Q'|\ge \eta|Q|.$$
\end{definition}

Figure \ref{anticarl} depicts an anti-Carleson stack of hight $3$ with $\eta\approx 0.7$. Note that the picture
is layered primarily according to the distance $d_{\mathcal F}$, so some longer intervals are drawn lower than some
shorter ones, which is opposite to how they would appear in the standard generation by generation layout of $\md$.

\begin{figure}[ht]
\begin{center}
\includegraphics[width=9cm,height=6cm]{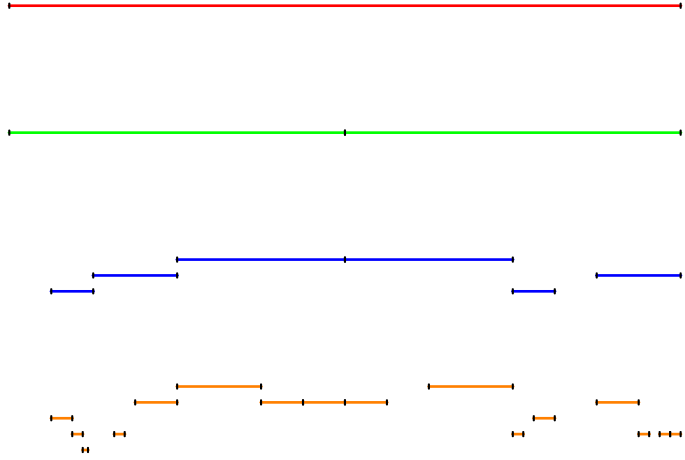}
\caption[]{An anti-Carleson stack $\mathcal F$ of height $3$ with the top interval $Q$ shown in red. The intervals
$Q'\in \mathcal F$ with $d_{\mathcal F}(Q,Q')=1,2,3$ are shown in green, blue, and orange respectively.}
\label{anticarl}
\end{center}
\end{figure}

Note that if $\s\subset\md$ is some family of cubes and we have an $\eta$-anti-Carleson stack ${\mathcal F}\subset \s$ of height $M$ with the top cube $Q$, then
$$\widetilde{\mathcal F}=\{Q'\in\s:Q'\subset Q, d_{\s}(Q,Q')\le M\}$$
is also an $\eta$-anti-Carleson stack of height $M$. This is an immediate consequence of 3 observations:
\begin{enumerate}[{1)}]
\item
since ${\mathcal F}\subset \s$, every cube $Q'\in {\mathcal F}$ with $d_{\mathcal F}(Q,Q')=M$ is contained in some $Q''\in\wt{\mathcal F}$
with $d_{\s}(Q,Q'')=M$;
\item
for every $Q''\in \wt{\mathcal F}$, $d_{\s}(Q,Q'')=d_{\wt{\mathcal F}}(Q,Q'')$;
\item
in any family of dyadic subcubes of $Q$, the cubes with some fixed distance to the top cube $Q$ are pairwise disjoint.
\end{enumerate}
Hence,
\begin{eqnarray*}
\eta|Q|&\le& \sum_{Q'\in {\mathcal F}, d_{\mathcal F}(Q,Q')=M}|Q'|=\Big|\bigcup_{Q'\in {\mathcal F}, d_{\mathcal F}(Q,Q')=M}Q'\Big|\\
&\le& \Big|\bigcup_{Q''\in\widetilde{\mathcal F}, d_{\s}(Q,Q'')=M}Q''\Big|=
\sum_{Q''\in\widetilde{\mathcal F}, d_{\wt{\mathcal F}}(Q,Q'')=M}|Q''|.
\end{eqnarray*}

Our next observation is that
for any fixed $\eta>0$, a Carleson family cannot contain an
$\eta$-anti-Carleson stack ${\mathcal F}$  of  arbitrarily large height $M$.
Indeed, since for $k\ge 1$, each cube $Q'$ with $d_{\mathcal F}(Q,Q')=k$ is contained in some $Q''$ with $d_{\mathcal F}(Q,Q'')=k-1$, we have
$$\sum_{Q'\in {\mathcal F}, d_{\mathcal F}(Q,Q')=k}|Q'|\ge \eta|Q|$$
for every $k=0,1,\dots,M$, from where it is clear that if ${\mathcal F}$ is contained in a  $\Lambda$-Carleson family of cubes, we must have  $(M+1)\eta\le \Lambda$. 

We will now prove a statement going in the opposite direction.

\begin{lemma}\label{noeta}
Let $\eta\in (0,1)$. If ${\mathcal S}\subset{\mathscr D}$ contains no $\eta$-anti-Carleson stack
of height $M$, then $\s$ is $\frac{M}{1-\eta}$-Carleson.
\end{lemma}

\begin{proof}
First of all, notice that ${\s}$ is $\Lambda$-Carleson if and only if every finite subsystem $\s'$ of $\s$ is $\Lambda$-Carleson.
Second, notice that the $\Lambda$-Carleson condition for $\s$ can be checked for cubes $Q\in\s$ only. Indeed, if it holds for every
cube $Q\in \s$, take any cube $R\in {\mathscr D}$ and consider the family ${\mathcal F}$ of all cubes $Q\in {\mathcal S}$ with $Q\subset R$
that are not contained in any other cube $Q'\in\s$ contained in $R$. Those cubes are disjoint and every other subcube of $R$ from $\s$
is contained in one of them, so we can write
\begin{eqnarray*}
\sum_{W\in\s,W\subset R}|W|&=&\sum_{Q'\in{\mathcal F}}\sum_{W\in\s,W\subset Q'}|W|\\
&\le& \sum_{Q'\in{\mathcal F}}\Lambda|Q'|
=\Lambda\Big|\bigcup_{Q'\in{\mathcal F}}Q'\Big|\le\Lambda|R|,
\end{eqnarray*}
establishing the $\Lambda$-Carleson property for every $R\in {\mathscr D}$.

Now take any finite $\s'\subset\s$. By the assumption of the lemma, $\s'$ contains no $\eta$-anti-Carleson stack of height $M$.
Observe that every finite family is Carleson (with Carleson constant not exceeding the number of cubes in the family, say).
Let $\Lambda$ be the best Carleson constant for $\s'$.
Let $Q\in \s'$. Let ${\mathcal F}_k(Q)$ be the set of all subcubes $Q'\in\s'$ of $Q$ with $d_{\s'}(Q,Q')=k$. Then
$\sum_{Q'\in {\mathcal F}_M}|Q'|\le \eta|Q|$ and every cube from any ${\mathcal F}_k$ with $k\ge M$ is contained in some $Q'\in {\mathcal F}_M$.
Thus, taking into account that the cubes in ${\mathcal F}_k$ are disjoint for each $k$, we get
\begin{eqnarray*}
\sum_{R\in \s',R\subset Q}|R|&=&\sum_{k=0}^{M-1}\sum_{R\in {\mathcal F}_k}|R|+\sum_{Q'\in {\mathcal F}_M}\sum_{R\in \s',R\subset Q'}|R|\\
&\le& M|Q|+\Lambda\sum_{Q'\in{\mathcal F}_M}|Q'|\le M|Q|+\eta\Lambda|Q|.
\end{eqnarray*}
Since this inequality holds for all $Q\in \s'$, we get $\Lambda\le M+\eta\Lambda$, i.e., $\Lambda\le \frac{M}{1-\eta}$.
\end{proof}

\subsection{Improving the Carleson constant by partitioning the family}
Some estimates are easier to carry out when the Carleson and sparseness constants are close to 1.
Usually, these estimates are of such nature that if they are obtained for $\s_1,\dots,\s_N$, then they trivially follow for the union
$\s=\s_1\cup\dots\cup\s_N$. So, one may be tempted to try to split a $\Lambda$-Carleson system with large $\Lambda$ into
several $\Xi$-Carleson systems with $\Xi<\Lambda$ (preferably just slightly above 1). The following lemma shows
that it is always possible.

\begin{lemma}
If $\s$ is a $\Lambda$-Carleson system and $m$ is an integer $\ge 2$, then $\s$ can be written as a union of $m$ systems $\s_j$,
each of which is $1+\frac{\Lambda-1}{m}$-Carleson.
\end{lemma}

\begin{proof}
As we saw above, we can check the Carleson condition for any family
$\s'\subset{\mathscr D}$ on the cubes $Q\subset {\s'}$ only.
Now consider the graph $\Gamma_{\s}$. It may be disconnected, but then every 2 cubes from different connected components are disjoint,
so it will suffice to split each connected component into $1+\frac{\Lambda-1}{m}$-Carleson systems separately. Thus, we can assume that $\Gamma_{\s}$
is connected, which means that any upward paths starting from two (or more)
cubes in $\s$ merge at some cube containing both (all) of them.

Now we can define an equivalence relation on $\s$ as follows:
take any two cubes $Q',Q''\in \s$ and find any cube $R\in\s$ that contains both of them. We say that $Q'$ is equivalent to $Q''$ if
$d_{\s}(Q',R)-d_{\s}(Q'',R)$ is divisible by $m$. Note that if we have some other cube $R'\in\s$ containing both $Q'$ and $Q''$, then either
$R\subset R'$, or $R'\subset R$, so
$$d_{\s}(Q',R')=d_{\s}(Q',R)\pm d_{\s}(R,R')$$ respectively and the same is true for $Q''$. Thus, the difference
$d_{\s}(Q',R)-d_{\s}(Q'',R)$ does not depend on the choice of $R\supset Q'\cup Q''$.

\begin{figure}[ht]
\begin{center}
\includegraphics[width=12cm,height=6cm]{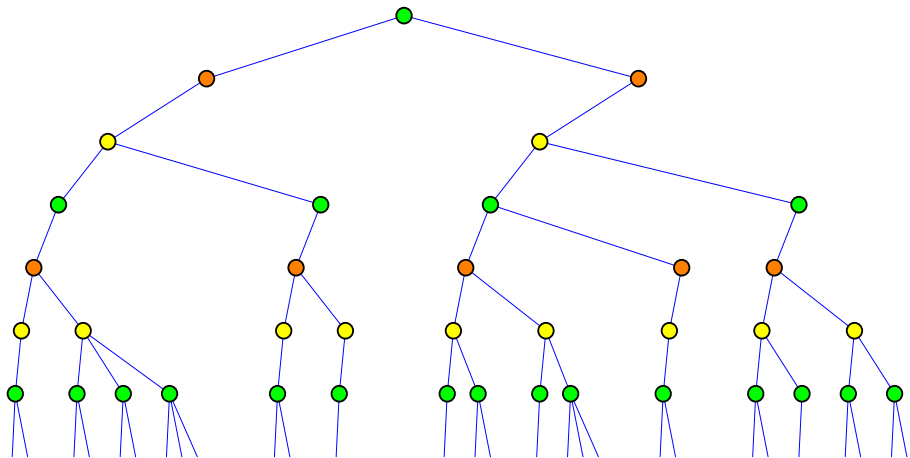}
\caption[]{The equivalence classes (with $m=3$) shown on the graph $\Gamma_{\s}$ in green, orange, and yellow.}
\label{carlsplit}
\end{center}
\end{figure}

Let $\s'\subset\s$ be one of the equivalence classes with respect to this relation. Take $Q\in \s'$ and consider the families
$${\mathcal F}_k=\{Q'\in\s:Q'\subset Q, d_{\s}(Q',Q)=k\}.$$ Then, by definition, the set $\{Q'\in \s':Q'\subset Q\}$ is the union
of ${\mathcal F}_0=\{Q\},{\mathcal F}_m, {\mathcal F}_{2m}$, etc. We also have
$$\sum_{Q'\in {\mathcal F}_k}|Q'|\ge \sum_{Q'\in {\mathcal F}_{ml}}|Q'|$$
for every $k=m(l-1)+1,\dots,ml$. Thus,
\begin{eqnarray*}
\Lambda|Q|\ge \sum_{k=0}^{\infty}\sum_{Q'\in {\mathcal F}_k}|Q'|&\ge& |Q|+m\sum_{l=1}^{\infty}\sum_{Q'\in{\mathcal F}_{ml}}|Q'|\\&=&|Q|+m\sum_{Q'\in {\mathcal S}',Q'\subsetneqq Q}|Q'|\,,
\end{eqnarray*}
whence
$$\sum_{Q'\in \s',Q'\subset Q}|Q'|\le |Q|+\frac{\Lambda-1}{m}|Q|,$$
proving the claim.
\end{proof}

\subsection{Augmentation of Carleson families}

\begin{lemma}\label{aug}
The augmentation of a $\Lambda_0$-Carleson family ${\mathcal S}$ by $\Lambda$-Carleson families ${\mathcal F}(Q)$ is a $\Lambda(\Lambda_0+1)$-
Carleson family.
\end{lemma}

\begin{proof}
Take any cube $Q\in {\mathscr D}$. The subcubes of $Q$ can appear in the augmented system $\widetilde\s$ either from $\mathcal F(\widehat Q_{\mathcal S})$ or from some ${\mathcal F}(Q')$ with $Q'\in {\mathcal D}(Q)\cap {\mathcal S}$. Since $\mathcal F(\widehat Q_{\mathcal S})$ is $\Lambda$-Carleson, we have
$$\sum_{R\in \mathcal F(\widehat Q_{\mathcal S}), R\subset Q}|R|\le \Lambda|Q|.$$
On the other hand,
$$\sum_{Q'\in\s,Q'\subset Q}\sum_{R\in \mathcal F(Q')}|R|\le \Lambda\sum_{Q'\in\s,Q'\subset Q}|Q'|\le \Lambda_0\Lambda|Q|\,.$$
Adding these inequalities, we get  the desired bound.
\end{proof}

Since we will often need to deal with stopping time  augmentations of Carleson families, it would be nice to find some easy to use
 criterion
for a condition ${\mathcal P}$ to be of Carleson type in the sense that for every $Q\in {\mathscr D}$, the family $\text{stop}(Q,\mathcal P)$ is Carleson
with uniformly controllable Carleson constant.

A simple sufficient condition is the following.

\begin{lemma}\label{stop}
Assume that for every cube $Q\in {\mathscr D}$ and any pairwise disjoint $Q_1',\dots,Q_N'\subset Q$ such that ${\mathcal P}(Q,Q_j')$ fails for each $j$, one has
$\sum_j|Q_j'|\le \eta|Q|$ with some $\eta<1$.
Then $\text{stop}(Q,\mathcal P)$ is $\frac{1}{1-\eta}$-Carleson for every $Q$.
\end{lemma}

\begin{proof}
This condition rules out $\eta$-anti-Carleson stacks of height 1 in $\text{stop}(Q,\mathcal P)$. It remains to apply Lemma \ref{noeta}.
\end{proof}

\section{From the theory to applications}
The Carleson families already lie on the border between ``basic" and ``intermediate" tools,
so we shall stop the general theory here and pass to an example of how dyadic techniques can be
applied to some problems in analysis that do not contain any dyadic lattices or anything else like that in their
original formulations. The particular application we chose for this short book is the weighted norm
inequalities for multilinear Calder\'on-Zygmund operators. We will restrict ourselves to the most classical part of the
theory and just prove the sharp bounds in terms of the joint Muckenhoupt norm of the weights alone leaving without
comment more refined estimates involving $A_{\infty}$-norms. Still, we hope that a reader
who does not know this particular subject in and out will find the exposition both accessible and
instructive.

\section{The multilinear Calder\'on-Zygmund operators}
The most classical, known, and loved singular integral operator is the Hilbert transform on the real line:
\begin{equation}\label{hilbert}
Hf(x)=\int_{{\mathbb R}}\frac{f(y)}{x-y}dy.
\end{equation}
The function $y\mapsto \frac{1}{x-y}$ is not integrable near $x$, so one can apply this formula verbatim
without any reservations only if $x$ is outside the (closed) support of $f$. Nevertheless, it turns out that some
tweaking (considering the principal value $\displaystyle\lim_{\e\to 0}\int_{\{y:|x-y|>\e\}}\frac{f(y)}{x-y}dy$, using Fourier
transform, etc.) allows one to show that $H$ can be made sense of as a nice bounded operator
on every $L^p({\mathbb R})$ with $1<p<\infty$, and even if $f\in L^1$ (which is the harshest case for making sense
of the integral formula (\ref{hilbert})), then $Hf$ is finite almost everywhere and satisfies
the weak type estimate
$$|\{x\in{\mathbb R}:|Hf(x)|>\a\}|\le \frac{C}{\a}\|f\|_{L^1}$$
for every $\a>0$.

The most straightforward ${\mathbb R}^n$ analogue of the Hilbert transform  is the Riesz transform:
$$Rf(x)=\int_{{\mathbb R}^n}\frac{x-y}{|x-y|^{n+1}}f(y)dy.$$
It is often very convenient to view $Rf$ as a vector-valued function, but one can always project to any
direction $e\in{\mathbb R}^n$ and consider
$$R_ef(x)=\int_{{\mathbb R}^n}\frac{\langle x-y,e\rangle}{|x-y|^{n+1}}f(y)dy$$
instead.

In general, we can consider arbitrary operators
$$Tf(x)=\int_{{\mathbb R}^n}K(x,y)f(y)dy$$
with the so-called Calder\'on-Zygmund kernels.
The development of the classical non-weighted theory of linear Calder\'on-Zygmund operators (even for the special case of the Hilbert transform)
is beyond the scope of this book. An interested reader can find a good exposition in \cite{GR,G,S}.
We will just note here that the main 3 properties that make the whole theory work are the following.

\medskip
\begin{enumerate}
\renewcommand{\labelenumi}{(\roman{enumi})}
\item
The scale and shift invariant estimate of the kernel:
$$|K(x,y)|\le \frac{C}{|x-y|^n}$$
(which makes $t^nK(tx,ty)$ as good (or as bad) as $K$, so if we conjugate the operator action by an affine change of variable $x\mapsto a+tx$, the entire
theory stays invariant).
\item
Some reasonable (and scale and shift invariant) continuity of the kernel away from the diagonal: if $Q$ is any cube, $x',x''\in Q$
and $y\not\in Q_{[t]}=c(Q)+t(Q-c(Q))$ with $t\ge 2$, then
$$|K(x',y)-K(x'',y)|\le\frac{C}{(t\ell_Q)^n}\rho(t^{-1}),$$
where $\rho:[0,+\infty)\to [0,+\infty)$ is some modulus of continuity (continuous, increasing and subadditive function such that $\rho(0)=0$)
that tends to $0$ not too slowly near $0$.
Note that it is almost always required that the same condition holds with the roles of $x$ and $y$ exchanged, but
 for our current  purposes
this part will suffice.
\item
Some cancellation in the kernel (antisymmetry, zero average, etc.). This part is crucial for the classical theory but here, where we just take
all statements of the classical non-weighted theory for granted, we are not concerned with it in any way.
\end{enumerate}

A multilinear ($m$-linear) Calder\'on-Zygmund operator takes in $m$ functions $f_1,\dots,f_m$ and its kernel $K(x,y_1,\dots,y_m)$ depends on $m+1$ variables
$x,y_1,\dots,y_m\in{\mathbb R}^n$. The result is defined as
$$T[f_1,\dots,f_m](x)=\int_{({\mathbb R}^n)^m}K(x,y_1,\dots,y_m)f(y_1)\dots f(y_m)dy_1\dots dy_m$$
when the integral makes sense (for a set $E\subset R^n$ we use a notation $E^m=E\times\dots\times E$) and the kernel is allowed to have a singularity only for $y_1=\dots=y_m=x$.
The scale and shift invariant bound will then be
$$|K(x,y_1,\dots,y_m)|\le \frac{C}{(\max\limits_i|x-y_i|)^{nm}}$$
and the corresponding continuity property can be stated as
$$|K(x',y_1,\dots,y_m)-K(x'',y_1,\dots,y_m)|\le \frac{C}{(t\ell_Q)^{mn}}\rho(t^{-1})$$
whenever $x',x''\in Q$,  $t\ge 2$, and there exists $i$ such that $y_i\not\in Q_{[t]}$.

The only nontrivial result of the classical theory we shall need below is the weak type bound
$$|\{x\in{\mathbb R}^n:|T[f_1,\dots,f_m](x)|>\a\}|\le C\left(\frac{1}{\a}\prod_{i=1}^m\|f_i\|_{L^1}\right)^{1/m}\,,$$
which we will just postulate.

At last, a few words should be said about the sense in which the action of $T$ in various weighted $L^p$ spaces of functions
is understood. The weak type property is, of course, formally contingent upon that $T[f_1,\dots,f_n]$ is well-defined
for any $f_i\in L^1$ as an almost everywhere finite function. Still, a function from an arbitrary weighted space does not
need to be even in $L^1_{\text{loc}}$ without some restrictions on the weight. What saves the day is that there is a linear
space of functions that is contained in $L^1$ as well as in  any weighted $L^p$  and dense there, namely the space $L^\infty_0$
of bounded
measurable functions with compact support. So, we will use it to take an easy way out: we will say that $T$ is bounded
as an operator from some product of weighted $L^p$-spaces to another weighted $L^p$-space if the corresponding
inequality for the norms holds for all test functions $f_i\in L^\infty_0$. Thus, we may always assume that all our test
functions below belong to $L^\infty_0$ though for most arguments it is not really necessary.

\section{Controlling values of $T[f_1,\dots,f_m]$ on a cube}
Our first goal will be to get an efficient pointwise estimate of $g=T[f_1,\dots,f_m]$ in terms of averages
$|f_i|_Q=\frac{1}{|Q|}\int_Q|f_i|$ of functions $f_i$ over various dyadic cubes.

The decay of the kernel $K(x,y_1,\dots,y_m)$ at infinity is too weak to make the trivial bound
$$|g(x)|\le \int_{({\mathbb R}^n)^m}|K(x,y_1,\dots,y_m)|\prod_{i=1}^m|f_i(y_i)|dy_1\dots dy_m$$
really useful even when the integral is finite. However, the continuity of the kernel in $x$ allows one to estimate the difference
of values of $g$ at two points $x',x''$ in some cube $Q$ by
\begin{eqnarray*}
&&\Big[|T[f_1\chi_{Q_{[2]}},\dots,f_m\chi_{Q_{[2]}}](x')|+|T[f_1\chi_{Q_{[2]}},\dots,f_m\chi_{Q_{[2]}}](x'')|\Big]\\
&&+\sum_{k=1}^{\infty}\int_{Q^m_{[2^{k+1}]}\setminus Q^m_{[2^k]}}|K(x',\vec y)-K(x'',\vec y)|\cdot |F(\vec y)|\,d\vec y=I_1+I_2,
\end{eqnarray*}
where $\vec y=(y_1,\dots,y_m)$, $F(\vec y)=f(y_1)\dots f(y_m)$.

By the continuity property of $K$,
$$I_2\le C\sum_{k=1}^{\infty}\rho(2^{-k})|F|_{Q^m_{[2^{k+1}]}}=C\sum_{k=1}^{\infty}\rho(2^{-k})\prod_{i=1}^m|f_i|_{Q_{[2^{k+1}]}}.$$

Note now that the weak type bound for $T$ implies that for the set
$$G_{\a}=\{x\in Q: |T[f_1\chi_{Q_{[2]}},\dots,f_m\chi_{Q_{[2]}}](x)|>\a\},$$
we have
$$|G_{\a}|\le\left(\frac{C}{\a}\prod_{i=1}^m|f_i|_{Q_{[2]}}\right)^{1/m}|Q|.$$
Hence, taking  $\a=C\la^{-m}\prod_{i=1}^m|f_i|_{Q_{[2]}}$, we obtain that   $|G_{\a}|\le \la|Q|$ and for every $x',x''\in E=Q\setminus G_{\a}$,
$$|I_1|\le 2C\la^{-m}\prod_{i=1}^m|f_i|_{Q_{[2]}}.$$

Therefore, we have proved that for every cube $Q$, there exists a set $E\subset Q$ with $|E|\ge (1-\la)|Q|$ and such that for all $x',x''\in E$,
$$|g(x')-g(x'')|\le C(\la)\sum_{k=0}^{\infty}\rho(2^{-k})\prod_{i=1}^m|f_i|_{Q_{[2^{k+1}]}}.$$
This phenomenon merits a few definitions and a separate discussion.

\section{Oscillation and $\la$-oscillation}
Let $f$ be a measurable function defined on some non-empty set $E$. The oscillation of $f$ on $E$ is just
$$\o(f;E)=\sup_Ef-\inf_Ef,$$
i.e., the length of the shortest closed interval containing $f(E)$. The oscillation is finite when $f$ is bounded on $E$.

Assume now that $f$ is measurable and finite almost everywhere on some measurable set $Q\subset {\mathbb R}^n$ of finite positive measure. Let $\la\in (0,1)$. The $\la$-oscillation
of $f$ on $Q$ is defined as
$$\o_{\la}(f;Q)=\inf\{\o(f;E): E\subset Q, |E|\ge (1-\la)|Q|\}.$$
Note that, unlike the usual oscillation $\o(f;Q)$, the $\la$-oscillation $\o_{\la}(f;Q)$ is finite for any measurable function $f$ that is finite almost everywhere on $Q$,
which makes it a much more flexible tool.

The first thing we note is that the infimum in the definition of $\o_{\la}(f;Q)$ is, actually, a minimum as the following lemma shows.

\begin{lemma}\label{osc}
For every $\la\in (0,1)$, there exists a set $E\subset Q$ of measure $|E|\ge (1-\la)|Q|$ with $\o_{\la}(f;Q)=\o(f;E)$.
\end{lemma}

\begin{proof}
There are subsets $E_k\subset Q$ with $|E_k|\ge (1-\la)|Q|$ and intervals $I_k=[a_k,a_k+\o(f;E_k)]\subset {\mathbb R}$ such that $f(E_k)\subset I_k$ and $\o(f;E_k)\downarrow \o_{\la}(f;Q)$.
The crucial point is that we cannot have more than $\frac{1}{1-\la}$ pairwise disjoint intervals $I_{k_j}$ with this property because otherwise
$$|Q|\ge |\cup_jf^{-1}(I_{k_j})|=\sum_j|f^{-1}(I_{k_j})|>\frac{1}{1-\la}(1-\la)|Q|=|Q|.$$
Thus, $\{a_k\}$ is a bounded sequence and, by Bolzano-Weierstrass, passing to a subsequence, if necessary, we can assume
without loss of generality that $a_k\to a$ as $k\to \infty$.

Now just put $E=\limsup_{k\to\infty}E_k=\cap_{n\ge 1}\cup_{k\ge n}E_k$. Note that $|E|\ge (1-\la)|Q|$ and for every $x\in E$ there is a subsequence $k_j$ such that
$$a_{k_j}\le f(x)\le a_{k_j}+\o(f;E_{k_j})$$ for all $j$, and, thereby,
$$a\le f(x)\le a+\o_{\la}(f;Q).$$
\end{proof}

Observe that the set $E\subset Q$ in Lemma \ref{osc} does not need to be unique. For example, let $Q=[0,1)$ and $f(x)=x$. Then $$\o_{\la}(f;Q)=1-\la=\o(f;E_{a,\la})$$ for all
intervals $E_{a,\la}=[a,a+1-\la)$ with $0\le a\le \la$.

\subsection{The oscillation chain bound for a measurable function}
We are now going to capitalize on the trivial observation that if $E_1,\dots, E_k$ is a chain of sets such that $E_j\cap E_{j+1}\not=\varnothing$ for all $j=1,\dots,k-1$, then for every $x\in E_1,y\in E_k$,
one has
$$|f(x)-f(y)|\le \sum_{j=1}^k\o(f;E_j).$$

\begin{theorem}\label{oscbound}
Let $f:{\mathbb R}^n\to {\mathbb R}$ be any measurable almost everywhere finite  function such that for every $\e>0$,
\begin{equation}\label{cond}
|\{x\in[-R,R]^n:|f(x)|>\e\}|=o(R^n)\,\,\text{as}\,\,R\to\infty.
\end{equation}
Then for every dyadic lattice ${\mathscr D}$ and every $\la\in (0,2^{-n-2}]$, there exists a regular%
\footnote{Recall that a family $\s\subset \md$ is regular if for each compact $K\subset\mathbb R^n$, there exists $Q\in\s$ containing $K$.}
 6-Carleson family $\s\subset {\mathscr D}$
(depending on $f$) such that
$$|f|\le \sum_{Q\in\s}\o_{\la}(f;Q)\chi_Q$$
almost everywhere.
\end{theorem}

\begin{proof}
For every cube $Q\in {\mathscr D}$, fix a set $E(Q)\subset Q$ such that $|E(Q)|\ge (1-\la)|Q|$ and $\o(f;E(Q))=\o_{\la}(f;Q)$.
We say that two cubes $Q\supset Q'$ are linked if $E(Q)\cap E(Q')\not=\varnothing$. Clearly, if $Q_N\supset\dots\supset Q_1$
is a chain of nested cubes in which $Q_{j+1}$ and $Q_j$ are linked for every $j=1,\dots,N-1$, then
$$\o(f;E(Q_1)\cup\dots\cup E(Q_N))\le \sum_{j=1}^N\o(f;E(Q_j)).$$

We call a family $\s$ linked if every two cubes $Q'\subset Q''$ in $\s$ with $d_{\s}(Q',Q'')=1$ are linked.
Assume that $\s$ is linked and regular. Then, starting from any cube $Q\in\s$, we can go up and enumerate all cubes from $\s$ we meet
on the way: $Q=Q_1\subset Q_2\subset\dots$. Take a point $x\in E(Q)$. For every $N\ge 1$, we have
$$|f(x)|\le \sum_{j=1}^{N-1}\o(f;E(Q_j))+\sup_{E(Q_N)}|f|.$$
However,  condition (\ref{cond}) implies that for every $\e>0$, $E(Q_N)$ intersects the set $\{|f|<\e\}$
if the cube $Q_N$ is sufficiently large, so in this case,
$$\sup_{E(Q_N)}|f|\le \e+\o(f;E(Q_N)).$$
Since $\e$ was arbitrary, we conclude that then $|f(x)|\le \sum_{j=1}^{\infty}\o(f;E(Q_j))$, which can be restated as
$$|f|\le \sum_{Q\in\s}\o_{\la}(f;Q)\chi_Q$$
on $\cup_{Q\in \s}E(Q)$.

Now it becomes clear how the system $\s$ should be constructed. We should just first make sure that the entire lattice ${\mathscr D}$
is linked and then to rarefy it as much as possible by removing the intermediate cubes in long linked descending chains
when a direct shortcut is available. The first requirement is equivalent to having each cube $Q$ linked to every of its children.
However, for every child $Q'$ of $Q$, we have $E(Q),E(Q')\subset Q$ and
$$|E(Q)|+|E(Q')|\ge(1-\la)(1+2^{-n})|Q|$$
so this requirement is automatically satisfied as soon as
$$(1-\la)(1+2^{-n})>1.$$

It may not be clear which cubes are essential and which are unnecessary for keeping ${\mathscr D}$ or its subsets linked if we look
at each cube alone and ask about its role in the entire lattice. However, we can certainly locate some cubes $Q'$ that seem like good candidates for removal if we start with some cube $Q\in\md$ and try to ``comb'' the lattice $\md$ starting with $Q$ and going down. Then, at each moment when we consider some particular cube $Q'\subset Q$ as a candidate for removal,
we still have the entire family $\mathcal D(Q')$ untouched and linked by its parent-child bonds.  Thus, at this moment, the maximal linked descending
chains starting at $Q$ and passing through $Q'$ will then pass through one of the children of $Q'$, so we can safely bypass
$Q'$ in all those chains if each child of $Q'$ is linked directly to $Q$ or to some intermediate cube between $Q$ and $Q'$ that
survived the previous combing.

Hence we just define the condition ${\mathcal P}(Q,Q')$ by saying that ${\mathcal P}(Q,Q')$ is satisfied if every child of $Q'\subset Q$ is linked to $Q$ (see Figure\ref{blobs}).

\begin{figure}[ht]
\begin{center}
\includegraphics[width=6cm,height=6cm]{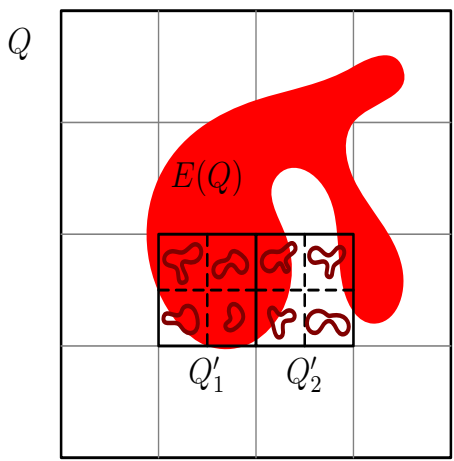}
\caption[]{$\mathcal P(Q,Q_1')$ holds because every child of $Q_1'$ is linked to $Q$, i.e., the corresponding sets encircled in brown
intersect $E(Q)$ (red), but $\mathcal P(Q,Q_2')$ fails, so the removal of $Q_2'$ may disconnect the chain linking $Q$ to the
bottom right child of $Q_2'$.}
\label{blobs}
\end{center}
\end{figure}

Combing the lattice down from some fixed cube in the manner described above
is exactly what our stopping time tool does, so it is natural now to
take a look at the families $\text{stop}(Q,\mathcal P)$. We would like them to be well-defined and Carleson, so we have
to check 2 things:

1) ${\mathcal P}(Q,Q)$ always holds. This was proved just a few lines ago when we talked about having the entire lattice linked.

2) ${\mathcal P}(Q,Q')$ is a Carleson-type condition. Indeed, suppose that $Q_j'$ is a family of pairwise disjoint subcubes of $Q$ such that
${\mathcal P}(Q,Q_j')$ fails for all $j$. Then for each $Q_j'$ there exists its child $Q_j''$ such that $E(Q_j'')\cap E(Q)=\varnothing$.
Since we still have $E(Q_j'')\subset Q_j''\subset Q$, we obtain
\begin{eqnarray*}
\sum_j|Q_j'|&=&2^n\sum_j|Q_j''|\le \frac{2^n}{1-\la}\sum_j|E(Q_j'')|\\
&\le& \frac{2^n}{1-\la}|Q\setminus E(Q)|\le \frac{2^n\la}{1-\la}|Q|\le \frac{1}{2}|Q|\,,
\end{eqnarray*}
provided that $\la\le\frac{1}{2^{n+1}+1}$.

Now just take any regular $2$-Carleson family (say, the family of all cubes in $\md$ containing the origin), and augment it by $\text{stop}(Q,{\mathcal P})$.
By Lemmas \ref{aug} and \ref{stop}, the resulting family $\s$ is $6$-Carleson. Also, it has
the property that ${\mathcal P}(\widehat Q_{\s},Q)$ holds for every $Q\in {\mathscr D}$. This can be used
in two ways.

Suppose that  $Q,Q'\in\s$, $Q\supset Q'$, and $d_{\s}(Q',Q)=1$. Then $Q$ is the
$\s$-roof of the parent $\widetilde Q$
of $Q'$, so ${\mathcal P}(Q,\widetilde Q)$ holds, i.e., $Q$ is linked to $Q'$ (as well as to the dyadic brothers of $Q'$, but that part is fairly useless).
Thus $\s$ is linked.

Another way to use this condition is to take any weak Lebesgue point $x\in {\mathbb R}^n$ of $f$ and, for fixed
$\e>0$, consider a cube $R\in {\mathscr D}$
containing $x$ such that
$$|\{y\in R:|f(y)-f(x)|>\e\}|<\la|R|\,.$$
Let $P$ be the parent of $R$.
Then $\widehat P_{\s}\in \s$, so we have
$$|f|\le \sum_{Q\in\s}\o_{\la}(f;Q)\chi_Q$$
on $E(\widehat P_{\s})$.
However, we also have $\widehat P_{\s}$ linked to all children of $P$ and, in particular to $R$, so
$$|f|\le \sum_{Q\in\s}\o_{\la}(f;Q)\chi_Q+\o_{\la}(f;R)$$
on $E(R)$.
Next, observe that $\o_{\la}(f;R)\le 2\e$
and the sets $E(R)\subset R$ and $\{y\in R:|f(y)-f(x)|\le\e\}$
intersect, provided that $\la<\frac 12$, so we finally conclude that
$$|f(x)|\le \sum_{Q\in\s}\o_{\la}(f;Q)\chi_Q+3\e\,.$$
Since $\e>0$ was arbitrary, the statement of the theorem follows.
\end{proof}

\section{From a sum over all cubes to a dyadic sum}
The result of Section 9 can  be restated as
\begin{equation}
\label{omegabd}
\o_{\la}(T[f_1,\dots,f_m];Q)\le C(\la)\sum_{k=0}^{\infty}\rho(2^{-k})\prod_{i=1}^m|f_i|_{Q_{[2^{k+1}]}}\,.
\end{equation}
Since the weak type bound implies that the measure of the set $\{x\in\mathbb R^n: |T[f_1,\dots,f_m]|>\e\}$ is finite
for every $\e>0$ and $f_1,\dots,f_m\in L^1$,
we can apply Theorem \ref{oscbound} to estimate $T[f_1,\dots,f_m]$ pointwise. A slight nuisance is that while in the resulting double sum $Q$ runs over
a Carleson subset $\s$ of a fixed dyadic lattice ${\md}$, the averages $|f_i|_{Q_{[2^{k+1}]}}$ are taken over non-dyadic
cubes. However, the Three Lattice Theorem allows one to dominate the right hand side of (\ref{omegabd}) by $3^n$ dyadic sums without any effort: note that for each $k$,
the cube $Q_{[2^{k+1}]}$ can be covered by a cube in one of the lattices $\md^{(j)}$ (see Remark \ref{remcor}) of comparable size and,
since the sidelengths of the cubes $Q_{[2^{k+1}]}$ increase as a geometric progression, each cube $R\in {\md^{(j)}}$ can
be used as an extension of $Q_{[2^{k+1}]}$ only finitely many times.
At last, the ratio $\ell_Q/\ell_R$ is comparable with $\ell_Q/\ell_{Q_{[2^{k+1}]}}=2^{-k-1}$, so, since the ratio $\rho(t)/\rho(ct)$
is uniformly bounded in $t>0$ for every fixed $c\in (0,1)$, we get
$$\o_{\la}(T[f_1,\dots,f_m];Q)\le C(n,\la)\sum_{j=1}^{3^n}\sum_{{R\in\md^{(j)}}\atop{R\supset Q}}\rho(\ell_Q/\ell_R)\prod_{i=1}^m|f_i|_R,$$
and, by Theorem \ref{oscbound},
$$|T[f_1,\dots,f_m]|\le C(n,\la)\sum_{j=1}^{3^n}\sum_{Q\in\s}
\sum_{{R\in\md^{(j)}}\atop{R\supset Q}}\rho(\ell_Q/\ell_R)\Big(\prod_{i=1}^m|f_i|_R\Big)\chi_Q$$
almost everywhere.

Now recall that for a fixed $j$, every cube $Q\in\md$ is contained in a unique cube $\widetilde Q_j\in \md^{(j)}$
with $\ell_{\widetilde Q_j}=3\ell_Q$. Since every $\widetilde Q_j$ can arise as an extension cube for at most $3^n$ cubes $Q$
and since the sets $E(Q)$ in the definition of the sparse family $\s$
can perfectly well serve to show the sparseness of the family $\s_j=\{\widetilde Q_j:Q\in\s\}$ only with a $3^n$ times
smaller sparseness constant, we conclude that for each set of functions $f_1,\dots,f_m$ and each $\la\in(0,2^{-n-2}]$, there exist $3^n$
dyadic lattices $\md^{(j)}$ and $3^n$  families $\s_j\subset \md^{(j)}$ each of which is $\Lambda$-Carleson with  $\Lambda=6\cdot 3^n$
such that
\begin{equation}\label{lastest}
|T[f_1,\dots,f_m]|\le C(n,\la)\sum_{j=1}^{3^n}\sum_{Q\in \s_j}
\sum_{{R\in\md^{(j)}}\atop{R\supset Q}}\rho(\ell_Q/\ell_R)\Big(\prod_{i=1}^m|f_i|_R\Big)\chi_Q
\end{equation}
pointwise.

\section{A useful summation trick}
In this section we describe a convenient way to estimate a double sum of the kind
$$\sum_{Q\in {\s}}\sum_{{R\in {\mathscr D}}\atop{R\supset Q}},$$
where $\md$ is a dyadic lattice and $\s$ is an arbitrary regular subset of $\md$.

The way the sum is written means that we take all cubes in $\s$ one by one and start looking up taking all dyadic cubes
$R$ on the way into account. We want to rewrite it so that the direction we look in changes from up to down
but the outer sum is still over $\s$, not over~${\mathscr D}$. This seems impossible until we realize that every cube $R\in {\mathscr D}$
has a roof $\widehat R_{\s}\in \s$, after which the task becomes trivial:
$$\sum_{Q\in {\s}}\sum_{{R\in {\mathscr D}}\atop{R\supset Q}}=\sum_{Q\in \s}\sum_{{U\in \s}\atop{U\supset Q}}\sum_{{R\in \mathscr D, \widehat R_{\s}=U}\atop{R\supset Q}}=\sum_{U\in \s}\sum_{{Q\in \s, R\in H_{\s}(U)}\atop{Q\subset R}}.$$

We shall now use this summation rearrangement with summands of the kind $a(d_{\md}(Q,R))\Psi(R)\chi_Q$,
where $a:{\mathbb Z}_+\to[0,\infty)$ is some reasonably fast decreasing function and $\Psi$ is some non-negative
quantity, which, after some change of notation, is exactly the kind of expression we have in (\ref{lastest}).

We observe, as before, that
$$\sum_{Q\in{\mathcal S}}\sum_{{R\in\md}\atop{R\supset Q}}a(d_{\md}(Q,R))\Psi(R)\chi_Q=
\sum_{U\in {\mathcal S}}\sum_{{Q\in\s, R\in H_{\s}(U)}\atop{Q\subset R}}a(d_{\md}(Q,R))\Psi(R)\chi_Q.$$
Since $R$ is now restricted to $H_{\s}(U)$, we can try to carry $\sup_{R\in H_{\s}(U)}\Psi(R)$ out of the inner sum
and look at the resulting bound
$$
\sum_{U\in\s}\Big(\sup_{R\in H_{\s}(U)}\Psi(R)\Big)
\sum_{{Q\in\s, R\in H_{\s}(U)}\atop{Q\subset R}}a(d_{\md}(Q,R))\chi_Q.
$$

It is now tempting to try to sum over $R$ in the inner sum taking into account that for a fixed $Q$, all distances $d_{\md}(Q,R),R\in\md,R\supset Q,$
are pairwise different, so we can get $\sum_{l=0}^{\infty}a(l)$ in the worst case scenario. This is not a bad idea and we will do exactly that with one
small but essential modification: we will notice that the condition $R\in H_{\s}(U)$ is much stronger than just $R\in\md$ and, since the cube $Q$
is also in $\s$ and $Q\subset R\subset U$, the distance $d_{\md}(R,Q)$ from $R\in H_{\s}(U)$ to $Q$ can never be less than $d_{\s}(U,Q)$.
Thus, splitting all possible cubes $Q\subset U$ in $\s$ according to their distance from $U$ in $\Gamma_{\s}$, we get the bound
$$\sum_{m\ge 0}
\sum_{{Q\in\s, R\in H_{\s}(U)}\atop{Q\subset R\subset U,d_{\s}(Q,U)=m}}a(d_{\md}(Q,R))\chi_Q
$$
for the inner sum.

Now, for fixed $m$, the sum
$$\sum_{{R\in H_{\s}(U)}\atop{R\supset Q}}a(d_{\md}(Q,R))$$
is still a sum of values of $a$ at distinct integers, but at this point we also know that none of them is less than $m$. Thus,
we get the bound $\sum_{l\ge m}a(l)$ for this sum instead of the sum over all $l\ge 0$. This may look like a small improvement
for each particular $Q$, but it becomes dramatic when we sum over $Q$ and take into account that for any fixed $m$,
$$\sum_{{Q\in\s, Q\subset U}\atop{d_{\s}(Q,U)=m}}\chi_Q\le \chi_U$$
(the cubes $Q$ in the summation are pairwise disjoint and are all contained in $U$).
We can now write
\begin{eqnarray*}
&&\sum_{m\ge 0}
\sum_{{Q\in\s, R\in H_{\s}(U)}\atop{Q\subset R\subset U,d_{\s}(Q,U)=m}}a(d_{\md}(Q,R))\chi_Q\le
\sum_{m\ge 0}
\sum_{{Q\in\s,d_{\s}(Q,U)=m}\atop{Q\subset U}}\Big(\sum_{l\ge m}a(l)\Big)\chi_Q\\
&&\le\sum_{m\ge 0}\Big(\sum_{l\ge m}a(l)\Big)\chi_U=\Big[\sum_{m\ge 0}(m+1)a(m)\Big]\chi_U.
\end{eqnarray*}
If $\beta=\sum_{m\ge 0}(m+1)a(m)<+\infty$, then we get the final estimate:
$$\sum_{Q\in{\mathcal S}}\sum_{{R\in\md}\atop{R\supset Q}}a(d_{\md}(Q,R))\Psi(R)\chi_Q\le \b
\sum_{U\in\s}\Big(\sup_{R\in H_{\s}(U)}\Psi(R)\Big)\chi_U.$$

\section{From Calder\'on-Zygmund operators to sparse operators}
The summation trick allows us to estimate $T[f_1,\dots,f_m]$ by
a sum of $3^n$ expressions of the kind
$$\sum_{Q\in\s}\Big(\sup_{R\in H_{\s}(Q)}\prod_{i=1}^m|f_i|_R\Big)\chi_Q$$
(one for each lattice $\md^{(j)}$)
under the additional assumption that $$\sum_{m=1}^{\infty} m\rho(2^{-m})<+\infty,$$ or, equivalently, $\int_0^1\rho(t)\log\frac{1}{t}\frac{dt}{t}<+\infty$. All classical operators have $\rho(t)=t^{\d}$ for some $\d>0$, so making this assumption hardly restricts the generality of the results we present.

Now we want to go one step further and replace the supremum by a single average. Of course, this is impossible if we keep
the original family $\s$ intact, but, since the only thing we know (or care) about is that $\s$ be Carleson, we can extend $\s$
in any way we want to a new family $\widetilde\s$ before applying the summation trick as long as the extension preserves
the Carleson property, and here is where the augmentation tool comes handy.

No matter how we extend, the top cube $Q$ of $H_{\widetilde S}(Q)$ will always remain in the supremum.
So, our task will be to make sure that for every $R\in H_{\widetilde\s}(Q)$, the product $\prod_{i=1}^m|f_i|_R$ is
controlled by $\prod_{i=1}^m|f_i|_Q$. This is exactly what the stopping time construction does if we choose the condition
$\mathcal P$ in an appropriate way. The most natural condition $\mathcal P(Q,Q')$ to introduce is the one that
$$\prod_{i=1}^m|f_i|_{Q'}\le K\prod_{i=1}^m|f_i|_Q$$
with some (large) $K\ge 1$. The products may be a bit unpleasant to handle directly, so we will ask for a bit more. Namely
we will say that ${\mathcal P}(Q,Q')$ holds if $|f_i|_{Q'}\le K^{1/m}|f_i|_Q$
for every $i=1,\dots,m$.
Then, for the augmented system $\widetilde S$, we will automatically have
$$\sup_{R\in H_{\widetilde \s}(Q)}\prod_{i=1}^m|f_i|_R\le K\prod_{i=1}^m|f_i|_Q$$
for every $Q\in\widetilde \s$ and our sum will get reduced to
$$\sum_{Q\in \widetilde\s}\prod_{i=1}^m|f_i|_Q\chi_Q.$$

Since we want to keep the family $\widetilde \s$ Carleson, we need to make sure that our condition
$\mathcal P$ is of Carleson type. Suppose that $Q_j'\subset Q$ are disjoint cubes for which $\mathcal P(Q,Q_j')$
fails. Then for each $j$, there exists $i(j)\in\{1,\dots,m\}$ with $|f_{i(j)}|_{Q_j'}> K^{1/m}|f_{i(j)}|_Q$.
Let $J_i$ be the set of $j$ for which $i(j)=i$. Then
\begin{eqnarray*}
K^{1/m}|f_i|_Q\sum_{j\in J_i}|Q_j'|&<& \sum_{j\in J_i}|f_i|_{Q_j'}|Q_j'|=\sum_{j\in J_i}\int_{Q_j'}|f_i|\\
&\le&\int_Q|f_i|=|f_i|_Q|Q|,
\end{eqnarray*}
whence
$\sum_{j\in J_i}|Q_j'|\le K^{-1/m}|Q|$ for each $i$  and $\sum_{j}|Q_j'|\le mK^{-1/m}|Q|.$
Thus, choosing $K=(2m)^m$, we get the desired Carleson type property for $\mathcal P$.

Now our strategy for the estimate is clear.
First, take each double sum
$$\sum_{Q\in \mathcal S_j}\sum_{{R\in\md^{(j)}}\atop{R\supset Q}}\rho(\ell_Q/\ell_R)\prod_{i=1}^m|f_i|_R\chi_Q,$$
and replace each Carleson family $\s_j$ by the augmented  family $\widetilde\s_j$, which is $2(6\cdot 3^n+1)$-Carleson by Lemmas \ref{aug} and \ref{stop}. We now have
$$\sup_{R\in H_{\wt\s_j}(Q)}\prod_{i=1}^m|f_i|_R\le (2m)^m\prod_{i=1}^m|f_i|_Q$$
for all $Q\in\wt\s_j$  and since $S_j\subset \widetilde S_j$, this replacement can only enlarge the double sum.
After that, we apply the summation trick, and, assuming that $\rho$ satisfies the condition $\int_0^1\rho(t)\log\frac{1}{t}\frac{dt}{t}<+\infty$, get the sum
$$\sum_{Q\in \widetilde\s_j}\prod_{i=1}^m|f_i|_Q\chi_Q$$
dominating the original double sum up to a constant factor.

This observation is worth singling out as a theorem. To state it clearer, let us make one definition first.

\begin{definition}\label{sparse}
Let $\md$ be any dyadic lattice and let $\s\subset\md$ be a Carleson family of cubes.
The $m$-linear sparse operator ${\mathcal A}_{\s,m}$ corresponding to the family $\s$ is defined by
$${\mathcal A}_{\s,m}[f_1,\dots,f_m]=\sum_{Q\in\s}\prod_{i=1}^m(f_i)_Q\chi_Q.$$
\end{definition}

\begin{theorem}\label{point}
Let $T$ be any $m$-linear Calder\'on-Zygmund operator with the modulus of continuity $\rho$ of the kernel
satisfying
$$\int_0^1\rho(t)\log\frac{1}{t}\frac{dt}{t}<+\infty.$$
Then, for every $f_1,\dots,f_m$, there exist $3^n$ dyadic lattices $\md^{(j)}$ and $14\cdot 3^n$-Carleson families
$\mathcal S_j\subset \md^{(j)}$ such that
$$|T[f_1,\dots,f_m]|\le C(T)\sum_{j=1}^{3^n}\mathcal A_{\s_j,m}[|f_1|,\dots,|f_m|]$$
almost everywhere. The constant $C(T)$ depends on $n,m,\rho$, and the constant in the weak type bound for $T$.
\end{theorem}

\section[Elementary weighted theory]{Elementary weighted theory of sparse multilinear forms}
The result of the previous section shows that, when building the theory of weighted norm inequalities for Calder\'on-Zygmund
operators, we can (at least, as a  first approximation) forget about the Calder\'on-Zygmund theory altogether and just
consider sparse operators ${\mathcal A}_{\s,m}$ instead. How much information will be lost on this way? This question is still
awaiting its answer because the fit is pretty tight but not perfectly tight. At this point we prefer to abstain from
discussing this issue in the hope that more light will be shed on it in the near future.

Now fix $\s$ and consider the simplest linear one-weight problem of finding a necessary and sufficient condition for ${\mathcal A}_{\s,1}$ to be bounded on $L^p(v)$, where $0< v\in L^1_{\text{loc}}$ is some weight and $p\in (1,+\infty)$.
Recall that
$$\|F\|_{L^p(v)}=\sup\left\{\Big|\int_{{\mathbb R}^n}Fgv\Big|:g\in L^{p'}(v), \|g\|_{L^{p'}(v)}\le 1\right\}$$
where $p'=\frac{p}{p-1}$ is the conjugate exponent. Thus, we are interested in establishing an inequality of the kind
\begin{eqnarray*}
\sum_{Q\in\s}f_Q(gv)_Q|Q|&=&\int_{{\mathbb R}^n}({\mathcal A}_{\s,1}f)gv\\
&\le&K\Big(\int_{{\mathbb R}^n}f^pv\Big)^{1/p} \Big(\int_{{\mathbb R}^n}g^{p'}v\Big)^{1/p'}
\end{eqnarray*}
for arbitrary non-negative functions $f,g$.

\subsection{The renormalization trick}
Now, we see one interesting thing: the integrals in the $L^{p'}$-norm of $g$ on the right and in the average $(gv)_Q$ are
taken with the same weight, but it is not so for $f$. Can we restore the symmetry between $f$ and $g$ in this respect?
The answer is ``yes". What we need to note is that if $w$ is any positive function, then $fw$ runs over all nonnegative
functions as $f$ does so, but the integrals on the left and on the right sides scale differently in $w$. So, our inequality
can be rewritten as
$$
\sum_{Q\in\s}(fw)_Q(gv)_Q|Q|
\le K\Big(\int_{{\mathbb R}^n}f^p(w^pv)\Big)^{1/p} \Big(\int_{{\mathbb R}^n}g^{p'}v\Big)^{1/p'}.
$$

Now, to restore a perfect symmetry between $f$ and $g$, we would like to find $w$ so that $w=w^pv$.
Solving this equation gives $w^{1/p}=wv^{1/p}$ or $w^{1-\frac 1p}v^{1-\frac1{p'}}=1$.
Thus, the one-weight linear inequality is equivalent to the 2-weight bilinear one
$$\sum_{Q\in\s}(f_1w_1)_Q(f_2w_2)_Q|Q|\le K
\Big(\int_{{\mathbb R}^n}f_1^{p_1}w_1\Big)^{1/p_1} \Big(\int_{{\mathbb R}^n}f_2^{p_2}w_2\Big)^{1/p_2}$$
with $\frac{1}{p_1}+\frac{1}{p_2}=1$ and the extra relation $w_1^{q_1}w_2^{q_2}\equiv 1$ ($q_i=1-\frac 1{p_i}$) for the weights.

The problem in this form, if we ignore for a moment all relations between $w_j$ and $p_j$, naturally generalizes to $m$ weights and functions.
To cover a few interesting cases that are not immediately transparent from the consideration of the one-weight linear problem
alone, we introduce a few more parameters and state the general multilinear weighted norm inequality as
\begin{equation}\label{first1}
\sum_{Q\in\s}\left(\prod_{i=1}^m(f_iw_i)_Q^{r_i}\right)|Q|\le K\prod_{i=1}^m\Big(\int_{{\mathbb R}^n}f_i^{p_i}w_i\Big)^{r_i/p_i}
\end{equation}
with $r_i>0,p_i>1$.
The reader who still wonders why it isn't
$$
\sum_{Q\in\s}\left(\prod_{i=1}^m(f_iu_i)_Q^{r_i}\right)|Q|\le K\prod_{i=1}^m\Big(\int_{{\mathbb R}^n}f_i^{p_i}v_i\Big)^{r_i/p_i}
$$
should recall the renormalization trick above and try carrying out the reduction of this seemingly more general version to (\ref{first1})
by himself.

\section{A digression: the dyadic maximal function}
For a weight $w$ and a measurable set $E\subset {\mathbb R}^n$
denote $w(E)=\int_Ew$.

Let $\md$ be a dyadic lattice. Given a weight $w$, define the weighted dyadic Hardy-Littlewood maximal operator  $M^{\md}_w$ by
$$M_{w}^{\mathscr{D}}f(x)=\sup_{Q\in {\mathscr{D}}: Q\ni x}\frac{1}{w(Q)}\int_Q|f|w.$$
The following result is known as the  Hardy-Littlewood dyadic maximal theorem.

\begin{theorem}\label{hldyad} The maximal operator $M_{w}^{\mathscr{D}}$ satisfies the following properties:
\begin{equation}\label{maxpr1}
w\{x\in {\mathbb R}^n: M_{w}^{\mathscr{D}}f(x)>\a\}\le \frac{1}{\a}\|f\|_{L^1(w)}\quad(\a>0)
\end{equation}
and
\begin{equation}\label{maxpr2}
\|M_{w}^{\mathscr{D}}f\|_{L^p(w)}\le \frac{p}{p-1}\|f\|_{L^p(w)}\quad(1<p\le \infty).
\end{equation}
\end{theorem}

\begin{proof}
Let $\mathcal F\subset \md$ be any finite family of cubes. Consider the restricted maximal function
$$
M_w^{\mathcal F}f= \begin{cases} \max\{\frac{1}{w(Q)}\int_Q|f|w: Q\in\mathcal F, Q\ni x\}, &x\in \cup_{Q\in\mathcal F}Q \\
0,& \text{otherwise}. \end{cases}
$$
By the monotone convergence theorem, it suffices to prove (\ref{maxpr1}) and (\ref{maxpr2}) for $M_w^{\mathcal F}$.

For $\a>0$, let
$$\O_{\a}=\{x\in {\mathbb R}^n: M_{w}^{\mathcal F}f(x)>\a\}.$$
Then $\O_{\a}$ is just the union of the maximal cubes $Q_j\in\mathcal F$ with the property that
$\int_{Q_j}|f|w >\a w(Q_j)$. Since $Q_j$ are disjoint, we get
\begin{equation}\label{woa}
w(\O_{\a})=\sum_jw(Q_j)\le \frac{1}{\a}\sum_j\int_{Q_j}|f|w=\frac{1}{\a}\int_{\O_{\a}}|f|w.
\end{equation}
This, obviously, implies the weak type bound for $M_w^{\mathcal F}$.

To get the $L^p(w)$-bound for $1<p<\infty$ (the remaining case $p=\infty$ is obvious),
just integrate (\ref{woa}) with the weight $p\a^{p-1}$:
\begin{eqnarray*}
\|M_{w}^{\mathcal F}f\|_{L^p(w)}^p&=&p\int_0^{\infty}\a^{p-1}w({\O_{\a}})d\a\le p\int_{0}^{\infty}\a^{p-2}\left(\int_{\O_{\a}}|f|w\right) d\a\\
&=&
p\iint_{\{(\alpha,x) :
0<\alpha<M_w^{\mathcal F}(x)\}}\alpha^{p-2}|f(x)|w(x)\,dx\,d\alpha\\
&=&
\frac{p}{p-1}\int_{{\mathbb R}^n}(M_{w}^{\mathscr{D}}f)^{p-1}|f|w\\
&\le& \frac{p}{p-1}\left(\int_{{\mathbb R}^n}
(M_w^{\mathcal F}f)^pw\right)^{\frac{p-1}{p}}\left(\int_{{\mathbb R}^n}|f|^pw\right)^{1/p}.
\end{eqnarray*}

Assuming that $f$ is bounded and compactly supported (so all integrals in the last inequality
are finite), we conclude that
$$
\left(\int_{{\mathbb R}^n}
(M_w^{\mathcal F}f)^pw\right)^{1/p}\le \frac{p}{p-1}\left(\int_{{\mathbb R}^n}|f|^pw\right)^{1/p}.
$$
For an arbitrary function $f\in L^p(w)$, just consider the truncated functions
$$
f_t(x) = \begin{cases} f(x), &|x|<t, |f(x)|<t\\
0& \text{otherwise} \end{cases}
$$
and use the monotone convergence theorem with $t\to\infty$.
\end{proof}

\section{A bound for weighted sparse multilinear forms}
We now return to the questions of when the inequality
\begin{equation}\label{first}
\sum_{Q\in\s}\left(\prod_{i=1}^m(f_iw_i)_Q^{r_i}\right)|Q|\le K\prod_{i=1}^m\Big(\int_{{\mathbb R}^n}f_i^{p_i}w_i\Big)^{r_i/p_i},
\end{equation}
 holds for all nonnegative $f_i\in L^{p_i}(w_i)$ with some $K>0$ and of how to find a good estimate for $K$ in terms of some
reasonable quantities that can be computed directly in terms of the weights $w_i$.
We will make an additional assumption that $\sum_{i=1}^m \frac{r_i}{p_i}=1$, which is a direct generalization of the relation $\frac 1{p_1}+\frac 1{p_2}=1$ in our main motivating bilinear
case discussed above.

The obvious necessary condition can be obtained by taking a cube $Q\in \s$ and putting $f_i=\chi_Q$ for all $i$. Then,
ignoring all terms in the sum on the left except the one corresponding to $Q$, we get
$$\prod_{i=1}^m(w_i)_Q^{q_i}\le K$$
with $q_i=r_i(1-\frac1{p_i})$ for all $Q\in\s$. We call this condition the joint $(\s;q_1,\dots,q_m)$-Muckenhoupt condition for
the system of weights $w_i$ and denote
$$
[w_1,\dots,w_m]_{\s;q_{1},\dots,q_m}=\sup_{Q\in\s}\prod_{i=1}^m(w_i)_Q^{q_i}\,.
$$
Note that if we want our bound to hold for {\em all} sparse multilinear forms corresponding to all possible
sparse families $\s$ with fixed sparseness constant $\eta>0$, we need our condition to hold for
all cubes $Q\subset\mathbb R^n$, so it is natural to define the full  joint $(q_1,\dots,q_m)$-Muckenhoupt condition as
$$
[w_1,\dots,w_m]_{q_{1},\dots,q_m}=\sup_{Q\subset\mathbb R^n, Q\text{ is a cube }}\prod_{i=1}^m(w_i)_Q^{q_i}<+\infty\,.
$$

The obvious way to use the sparseness of $\s$ is to switch to the disjoint sets $E(Q)$. Note that we have no guarantee that $f_i$
are not small on $E(Q)$. However, we can be sure that the dyadic maximal functions $M^{\md}_{w_i}f_i$ satisfy
$M^{\md}_{w_i}f_i\ge\frac{(f_iw_i)_Q}{(w_i)_Q}$ on $E(Q)$.
Thus, using that $\sum_{i=1}^m\frac{r_i}{p_i}=1$ again, we obtain the trivial H\"older bound
\begin{eqnarray}
&&\prod_{i=1}^m\Big(\int_{{\mathbb R}^n}(M^{\md}_{w_i}f_i)^{p_i}w_i\Big)^{r_i/p_i}\ge \prod_{i=1}^m\Big(\sum_{Q\in\s}\int_{E(Q)}
(M^{\md}_{w_i}f_i)^{p_i}w_i\Big)^{r_i/p_i}\label{second}\\
&&\ge \prod_{i=1}^m\Big(\sum_{Q\in\s}\Big(\frac{(f_iw_i)_Q}{(w_i)_Q}\Big)^{p_i}w_i(E(Q))\Big)^{r_i/p_i}\nonumber\\
&&\ge \sum_{Q\in{\mathcal S}}\prod_{i=1}^m\frac{(f_iw_i)_Q^{r_i}}{(w_i)_Q^{r_i}}w_i(E(Q))^{r_i/p_i}.\nonumber
\end{eqnarray}

By the Hardy-Littlewood dyadic maximal theorem, we see that the left hand side of (\ref{second}) is dominated by the right hand side of (\ref{first}).
A good chunk of the classical theory of weighted norm inequalities can be derived from comparing the right hand side of (\ref{second}) with the left hand side of
(\ref{first}).

The direct multilinear analogue of the linear one-weight case is the situation where the weights $w_i$ are related by $\prod_{i=1}^mw_i^{q_i}\equiv 1$.
Write
\begin{eqnarray*}
\left(\prod_{i=1}^m(f_iw_i)_Q^{r_i}\right)|Q|&=&\prod_{i=1}^m\frac{(f_iw_i)_Q^{r_i}}{(w_i)_Q^{r_i}}w_i(E(Q))^{r_i/p_i}\times\prod_{i=1}^m(w_i)_Q^{q_i}\\
&\times&
\left(\prod_{i=1}^m\Big(\frac{(w_i)_Q}{w_i(E(Q))}\Big)^{r_i/p_i}\right)|Q|.
\end{eqnarray*}

The first factor is just the quantity on the right hand side of (\ref{second}); the second factor is bounded by $[w_1,\dots,w_m]_{\s;q_1,\dots,q_m}$, so
the whole game is in getting a decent bound for the last factor.
Fortunately, in the case under consideration, this factor cannot be too large. Put $q=\sum_{i=1}^m q_i$. Then, by H\"older's inequality,
$$\prod_{i=1}^mw_i(E(Q))^{q_i/q}\ge \int_{{E(Q)}}\prod_{i=1}^mw_i^{q_i/q}=|E(Q)|\ge \eta|Q|.$$
Thus, we always have
\begin{equation}\label{prodi}
\prod_{i=1}^m\Big(\frac{w_i(E(Q))}{|Q|}\Big)^{q_i}\ge \eta^q.
\end{equation}
Denoting $\a_i(Q)=\frac{w_i(E(Q))}{|Q|}$ and $\b_i(Q)=\frac{w_i(Q)}{|Q|}\ge\a_i(Q)$,
we see that the product we need to estimate can be written as $\prod_{i=1}^m\Big(\frac{\b_i(Q)}{\a_i(Q)}\Big)^{r_i/p_i}$ and
the Muckenhoupt condition along with (\ref{prodi}) implies that
\begin{eqnarray*}
\prod_{i=1}^m\Big(\frac{\b_i(Q)}{\a_i(Q)}\Big)^{q_i}&=&\prod_{i=1}^m\Big(\frac{|Q|}{w_i(E(Q))}\Big)^{q_i}\prod_{i=1}^m(w_i)_Q^{q_i}\\
&\le&\frac{1}{\eta^q}[w_1,\dots,w_m]_{\s;q_1,\dots,q_m}.
\end{eqnarray*}
Since each ratio $\frac{\b_i(Q)}{\a_i(Q)}$ is at least 1, we conclude that
$$\prod_{i=1}^m\Big(\frac{\b_i(Q)}{\a_i(Q)}\Big)^{r_i/p_i}\le (\eta^{-q}[w_1,\dots,w_m]_{\s;q_1,\dots,q_m})^{\max\limits_{1\le i\le m}\frac{r_i}{p_iq_i}}$$
and, thereby,
\begin{equation}\label{kbound}
K\le \eta^{-\tau}\prod_{i=1}^m\Big(\frac{p_i}{p_i-1}\Big)^{r_i}[w_1,\dots,w_m]_{\s;q_1,\dots,q_m}^{\gamma}
\end{equation}
where
$$
\gamma=\max_{1\le i\le m}\big(1+\frac{r_i}{p_iq_i}\big)
=\max_{1\le i\le m}\big(1+\frac{r_i}{p_ir_i(1-\frac 1{p_i})}\big)=\max_{1\le i\le m}p_i'
$$
and
$$
\tau=q\max_{1\le i\le m}\frac{r_i}{p_iq_i}=q\max_{1\le i\le m}\frac{r_i}{p_ir_i(1-\frac 1{p_i})}
=q\max_{1\le i\le m}\frac1{p_i-1}\,.
$$

Observe that the $L^p(w)$ operator norm of $M_w^{\md}$ blows up as $p\to 1$, but approaches $1$ as $p\to+\infty$.
Thus, despite the argument above has been carried out for finite $p$, the same inequality
\begin{eqnarray*}
&&\sum_{Q\in\s}\prod_{i=1}^m(f_iw_i)_Q^{r_i}|Q|\\
&&\le\eta^{-\tau}\prod_{i=1}^m\Big(\frac{p_i}{p_i-1}\Big)^{r_i}
[w_1,\dots,w_m]_{\s;q_1,\dots,q_m}^{\gamma}\prod_{i=1}^m\|f_i\|_{L^{p_i}(w_i)}
\end{eqnarray*}
can be obtained for the case when some of $p_i$ are equal to $+\infty$ either by repeating the proof, or just by passing to the limit.

Note that the passage to the limit is not completely trivial because one has to be careful to make sure that all the related quantities, indeed, do tend to what one wants them to tend to and that the conditions  $\sum_{i=1}^m\frac {r_i}{p_i}=1$ and $\prod_{i=1}^m w_i^{q_i}\equiv 1$ hold all the way through. The easiest way to ensure that all the limits
exist and  are correct is to assume that
$\s$ is finite and all test functions are compactly supported and bounded, and to let the monotone convergence theorem  take care of the general case.
To  preserve the conditions, one can put, for example,
$\frac{1}{p_i(t)}=\frac 1{p_i}+t\left(1-\frac 1{p_i}\right)$ and $r_i(t)=\psi(t)r_i$ with $\psi(t)=\left(1-t+t\sum_{i=1}^m r_i\right)^{-1}$ and  $t$ tending to $0+$.

\section{The weighted $L^p$-estimates}
We are now ready to consider the action of ${\mathcal A}_{\s,m}$ from
$L^{p_1}(v_1)\times\dots\times L^{p_{m}}(v_{m})$ to $L^p(v)$. Since we are ultimately aiming at transferring the result we will obtain
to the case of multilinear Calder\'on-Zygmund operators, we will be interested in the uniform bounds for the whole family of sparse operators
with fixed sparseness constant $\eta>0$  here, rather
than in the action of each individual one like it was in the previous section.

Note that $\mathcal S=\md_k$ is a $1$-Carleson family, and that the corresponding
 sparse operators ${\mathcal A}_{{\md}_k,m}$
converge to the ``trivial" multilinear operator
$$T_0[f_1,\dots,f_{m}]=f_1\dots f_{m}$$
as $k\to\infty$.
Thus, to get a natural setup, we want to be sure that, at least, $T_0$ acts as a bounded multilinear operator from
$L^{p_1}(v_1)\times\dots\times L^{p_{m}}(v_{m})$ to $L^p(v)$.

Observe that the product of $m$ functions $f_i\in L^{p_i}(v_i)$ of norm $1$ is an arbitrary  function from the unit ball of $L^p(v)$ with $p$
given by $\frac{1}{p}=\sum_{i=1}^{m}\frac{1}{p_i}$ and $v=\prod_{i=1}^{m}v_i^{p/p_i}$.
Thus, the best we can hope for is that the entire family of sparse operators $A_{\s,m}$
acts from $L^{p_1}(v_1)\times\dots\times L^{p_{m}}(v_{m})$ to $L^p(v)$ with these particular
$p$ and $v$. If the families of weights $v_i$ and powers $p_i$ satisfy this property, then any other
statement about the operator boundedness of ${\mathcal A}_{\s,m}$ from
$L^{p_1}(v_1)\times\dots\times L^{p_{m}}(v_{m})$ to some weighted $L^p$-space either follows from it, or is false.

We will now try to reduce this case to the estimates for weighted multilinear forms obtained in the previous
section using duality and renormalization. Since the duality trick increases the number of functions
and integrations by $1$, the action of ${\mathcal A}_{\s,m}$ in weighted spaces is equivalent to the
boundedness of some weighted $m+1$-form. Note also that for $m>1$, we can easily have $p<1$ even
under our assumption that all $p_i>1$,
so the duality argument has to be applied with some caution.

Assume first that $p>1$. Then we can take $p_{m+1}=p', r_i=1$ (so $q_i=1-\frac{1}{p_i}$ for $i\le m$ and $q_{m+1}=\frac{1}{p}$). Note that with this choice, $\sum_{i=1}^{m+1}\frac{r_i}{p_i}=\frac 1p+\frac 1{p'}=1$.
The renormalization trick implies that
$\|{\mathcal A}_{\s,m}\|_{L^{p_1}(v_1)\times\dots\times L^{p_{m}}(v_{m})\to L^p(v)}\le K$ if and only if
the estimate
$$
\int_{{\mathbb R}^n}{\mathcal A}_{\s,m}[f_1w_1,\dots,f_{m}w_{m}]f_{m+1}w_{m+1}
\le
K\prod_{i=1}^{m+1}\|f_i\|_{L^{p_i}(w_i)},
$$
holds for all non-negative functions $f_i\in L^{p_i}(w_i)$
with $w_i^{1-p_i}=v_i\,(i=1,\dots, m)$ and $w_{m+1}=v$. Observe now that in this case,
the relation $v=\prod_{i=1}^{m}v_i^{p/p_i}$
becomes $w_{m+1}=\prod_{i=1}^{m}w_i^{\frac{p}{p_i}-p}$ or $w_{m+1}^{1/p}\prod_{i=1}^{m}w_i^{1-1/p_i}\equiv 1$,
which is exactly the additional relation we imposed on the weights when proving (\ref{kbound}).
Thus, the results of the previous section yield
$$
\|{\mathcal A}_{\s,m}\|_{L^{p_1}(v_1)\times\dots\times L^{p_{m}}(v_{m})\to L^p(v)}\le
C(\eta,p_i)[w_1,\dots,w_{m+1}]_{q_1,\dots,q_{m+1}}^{\gamma}
$$
with
$$\gamma=\max_{1\le i\le m+1}p_i'=\max(p_1',\dots,p_{m}',p).$$

In the case $p\le 1$, we put $p_{m+1}=\infty, r_i=p$ for $i=1,\dots,m$, and $r_{m+1}=1$. We still have
$\sum_{i=1}^{m+1}\frac{r_i}{p_i}=\sum_{i=1}^{m}\frac p{p_i}+\frac 1{\infty}=1+0=1$.
Now $q_i=p(1-\frac 1{p_i})$ for $i\le m,$
$q_{m+1}=1$, and we see that to apply the results of the previous section, we need the relation $w_{m+1}\prod_{i=1}^{m}w_i^{(1-\frac 1{p_i})p}\equiv 1$, which is the same as before.
Our inequality then becomes
\begin{eqnarray*}
&&\sum_{Q\in \s}\left(\prod_{i=1}^{m}(f_iw_i)_Q^p(f_{m+1}w_{m+1})_Q\right)|Q|\\
&&\le C(\eta,p_i)
[w_1,\dots,w_{m+1}]_{q_1,\dots,q_{m+1}}^{\gamma}\Big(\prod_{i=1}^{m}\|f_i\|_{L^{p_i}(w_i)}^p\Big)\|f_{m+1}\|_{L^{\infty}(w_{m+1})}.
\end{eqnarray*}
or, equivalently,
\begin{eqnarray*}
&&\int {\mathcal A}_{p,\s,m}[f_1w_1,\dots,f_{m}w_{m}]w_{m+1}\\
&&\le
C(\eta,p_i)[w_1,\dots,w_{m+1}]_{q_1,\dots,q_{m+1}}^\gamma\prod_{i=1}^{m}\|f_i\|^p_{L^{p_i}(w_i)}
\end{eqnarray*}
where
$$
{\mathcal A}_{p,\s,m}[f_1,\dots,f_{m}]=\sum_{Q\in\s}\prod_{i=1}^{m}(f_i)_Q^p\chi_Q\,.
$$
Since $p\le 1$, we have ${\mathcal A}_{\s,m}[f_1,\dots,f_{m}]^p\le {\mathcal A}_{p,\s,m}[f_1,\dots,f_{m}]$ pointwise
for any non-negative functions $f_i$. Hence, in this case, we get
\begin{eqnarray*}
&&\|{\mathcal A}_{\s,m}[f_1w_1,\dots,f_{m}w_{m}]\|_{L^p(w_{m+1})}\\
&&\le
C(\eta,p_i)[w_1,\dots,w_{m+1}]_{q_1,\dots,q_{m+1}}^{\gamma/p}\prod_{i=1}^{m}\|f_i\|_{L^{p_i}(w_i)}
\end{eqnarray*}
with
$$\gamma=\max(p_1',\dots,p_{m}',1)\,.$$

Note that the algebraic
expressions for the numbers $q_i$ in the case $p\le 1$ are exactly $p$ times the corresponding expressions in
the case $p>1$, so writing everything directly in terms of $p_i$, we see that the Muckenhoupt factor is
$$
[w_1,\dots,w_{m+1}]_{1/p_1',\dots,1/p_{m}',1/p}^{\max(p_1',\dots,p_{m}',p)}
$$
in both cases.

By Theorem \ref{point}, these estimates can be immediately extended to multilinear Calder\'on-Zygmund operators and we obtain the following

\begin{theorem}
\label{mainthm}
Let $T$ be a multilinear Calder\'on-Zygmund operator with the modulus $\rho$ of continuity
of the kernel satisfying
$$\int_0^1\rho(t)\log\frac 1t \frac{dt}{t}<+\infty.$$
Let $v_i$ be any weights and $p_i>1$ be
any numbers ($i=1,\dots,m$). Define $p>0$ by $\frac 1p=\frac{1}{p_1}+\dots+\frac 1{p_{m}}$ and put $v=\prod_{i=1}^{m}v_i^{p/p_i}$,
$w_{m+1}=v$, $w_i=v_i^{-1/(p_i-1)}$ ($i=1,\dots,m$). Then
\begin{eqnarray*}
&&\|T\|_{L^{p_1}(v_1)\times\dots\times L^{p_{m}}(v_{m})\to L^p(v)}\le C(T,p_i)[w_1,\dots,w_{m+1}]_{1/p_1',\dots,1/p_{m}',1/p}^{\max(p_1',\dots,p_{m}',p)}\\
&&=C(T,p_i)\left[\sup_{Q\subset \mathbb R^n,Q\text{ \rm is a cube}} v_Q^{1/p}\prod_{i=1}^{m} (v_i^{-p_i'/p_i})_Q^{1/p_i'}\right]^{\max(p_1',\dots,p_{m}',p)}\,.
\end{eqnarray*}
\end{theorem}

The only remark that remains to make to juxtapose our statement of this theorem with the way it is usually
written in the literature is that $\sup_{Q\subset \mathbb R^n,Q\text{ \rm is a cube}} v_Q^{1/p}\prod_{i=1}^{m} (v_i^{-p_i'/p_i})^{1/p_i'}$ is usually denoted by
$[\vec v]_{A_{\vec P}}^{1/p}$. In particular, in the linear case, when $\vec v=v, \vec P=p$, we obtain
$$\|T\|_{L^p(v)\to L^p(v)}\le C(T,p)[v]_{A_p}^{\max(p'/p,1)}=C(T,p)[v]_{A_p}^{\max(\frac{1}{p-1},1)}\,.$$

\vskip 1cm
\begin{center}
{\bf Historical notes}
\end{center}

Multilinear Calder\'on-Zygmund operators in the form considered in this book probably first appeared in Coifman and Meyer \cite{CM1}, but their systematic study was started by Grafakos and Torres \cite{GT1}.

For a classical (linear) theory of $A_p$ weights we refer the reader to Garc\'ia-Cuerva and Rubio de Francia \cite{GR} and to Grafakos \cite{G}. Muckenhoupt families of weights for multilinear operators
were introduced in Lerner et al. \cite{LOPTT}.

The main result of Section 10,
Theorem \ref{oscbound}, has been well known in the local form, i.e., for functions defined on some cube $Q\subset {\mathbb R}^n$ (see, for instance, Hyt\"onen \cite{Hyt2}) .

The linear case of Theorem \ref{mainthm} with $p=2$ has become known as the ``$A_2$ conjecture''. It was first proved
for the Hilbert transform by Petermichl \cite{P} and for general Calder\'on-Zygmund operators by Hyt\"onen \cite{Hyt1}.
The approach based on dyadic sparse operators is due to Lerner \cite{L2},
where a weaker form of Theorem \ref{point} with the Banach function space norm estimate instead of the pointwise estimate was obtained. 

An alternative proof of Theorem \ref{point} was given by Conde-Alonso and Rey in \cite{CG}. 
Later, Lacey \cite{La} relaxed the condition on the modulus of continuity $\rho$ to $\int_0^1\rho(t)\frac{dt}{t}<\infty$. 

The bounds for weighted sparse multilinear forms in Section 16 were obtained by Li, Moen, and Sun \cite{LMS}. Their ultimate goal in that paper was also to derive Theorem \ref{mainthm}, but, lacking the pointwise bound given by Theorem \ref{point}, they
 were forced to stay within the category of Banach function spaces and to introduce the extra assumption $p\ge~ 1$.

\end{document}